\documentclass[12pt,leqno,draft]{article} 
\usepackage{amsmath,amssymb,amsthm,amsfonts,bm}
\usepackage{enumerate,color}
\makeatletter
\def\@cite#1#2{[{{\bfseries #1}\if@tempswa , #2\fi}]}
\renewcommand{\section}{%
\@startsection{section}{1}{\z@}
{0.5truecm plus -1ex minus -.2ex}%
{1.0ex plus .2ex}{\bfseries\large}}
\def\@seccntformat#1{\csname the#1\endcsname.\ }
\makeatother

\numberwithin{equation}{section} 
\newtheorem{thm}{Theorem}[section]

\newtheorem{lem}[thm]{Lemma}

\theoremstyle{definition}
\newtheorem{df}{Definition}[section]
\newtheorem{remark}{Remark}[section]

\newtheorem*{prth1.1}{Proof of Theorem 1.1}
\newtheorem*{prth1.2}{Proof of Theorem 1.2}
\newtheorem*{prth1.3}{Proof of Theorem 1.3}
\newtheorem*{prth1.4}{Proof of Theorem 1.4}

\newcommand{\ep}{\varepsilon}

\newcommand{\RN}{\mathbb{R}^N}

\begin{document}
\footnote[0]
    {2010{\it Mathematics Subject Classification}\/. 
    Primary: 35K59. 
    }
\footnote[0]
    {{\it Key words and phrases}\/: 
    nonlocal Cahn--Hilliard systems; unbounded domains. 
    }
\begin{center}
    \Large{{\bf 
Existence and energy estimates of weak solutions for \\
nonlocal Cahn--Hilliard equations\\ on unbounded domains 
           }}
\end{center}
\vspace{5pt}
\begin{center}
    Shunsuke Kurima\footnote{Partially supported 
    by JSPS Research Fellowships for Young Scientists (No.\ 18J21006).} \\
    \vspace{2pt}
    Department of Mathematics, 
    Tokyo University of Science\\
    1-3, Kagurazaka, Shinjuku-ku, Tokyo 162-8601, Japan\\
    {\tt shunsuke.kurima@gmail.com}\\
    \vspace{2pt}
\end{center}
\begin{center}    
    \small \today
\end{center}

\vspace{2pt}
\newenvironment{summary}
{\vspace{.5\baselineskip}\begin{list}{}{%
     \setlength{\baselineskip}{0.85\baselineskip}
     \setlength{\topsep}{0pt}
     \setlength{\leftmargin}{12mm}
     \setlength{\rightmargin}{12mm}
     \setlength{\listparindent}{0mm}
     \setlength{\itemindent}{\listparindent}
     \setlength{\parsep}{0pt}
     \item\relax}}{\end{list}\vspace{.5\baselineskip}}
\begin{summary}
{\footnotesize {\bf Abstract.}
This paper considers the initial-boundary value problem for 
the nonlocal Cahn--Hilliard equation 
\begin{equation*}
    \partial_t\varphi + (-\Delta+1)(a(\cdot)\varphi -J\ast\varphi + G'(\varphi)) = 0
         \quad \mbox{in}\ \Omega\times(0, T) 
\end{equation*}
in an {\it unbounded} domain $\Omega \subset \RN$ 
with smooth bounded boundary, where 
$N\in\mathbb{N}$, $T>0$, and 
$a(\cdot), J, G$ are given functions. 
In the case that $\Omega$ is a bounded domain and 
$-\Delta+1$ is replaced with $-\Delta$, 
this problem has been studied 
by using a Faedo--Galerkin approximation scheme 
considering the compactness of the Neumann operator $-\Delta+1$ 
(cf.\ \cite{CFG-2012, GG-2014}). 
However, the compactness of the Neumann operator $-\Delta+1$ 
breaks down   
when $\Omega$ is an {\it unbounded} domain.  
The present work establishes existence and energy estimates of weak solutions 
for the above problem 
on an {\it unbounded} domain. }
\end{summary}
\vspace{10pt}

\newpage

\section{Introduction and results} \label{Sec1}

A well-known model 
for the phase separation in a binary alloys system 
is called the Cahn--Hilliard model 
which was produced in \cite{CH-1958}. 
The following equation is the basic form of this model:   
\begin{equation*}
 \begin{cases}\label{E0}\tag{E0}
 \partial_{t}\varphi  - \mbox{div}(\kappa(\varphi)\nabla\mu) = 0 
 & \mbox{in}\ \Omega\times(0, T), \\[-0mm]
 \mu = -\delta\Delta\varphi + \frac{1}{\delta}G'(\varphi)  
         & \mbox{in}\ \Omega\times(0, T), 
 \end{cases}
\end{equation*}
where $\Omega \subset \mathbb{R}^{d}$ ($d=2, 3$) is a bounded domain, 
$\kappa$ is the mobility coefficient, $\delta>0$ is a given small constant 
with respect to the thickness of the interface, 
and $G'$ is the first derivative of a double well potential $G$. 
The above chemical potential $\mu$ 
is the Fr\'echet derivative of the free energy functional (see \cite{CH-1958})  
$$
E_{0}(\varphi) = 
\int_{\Omega} \left(\frac{\delta}{2}|\nabla\varphi|^2 
                                             + \frac{1}{\delta}G(\varphi) \right)\,dx. 
$$

Giacomin--Lebowitz \cite{GL0} 
observed that 
the Cahn--Hilliard equation \eqref{E0} has no microscopic derivation. 
They started from the microscopic viewpoint 
and proposed a macroscopic equation 
describing the phase segregation phenomena (see \cite{GL1, GL2}).
This equation is a nonlocal Cahn--Hilliard equation. 
Also, there are nonlocal Cahn--Hilliard--Navier--Stokes systems 
which model the evolution of an isothermal mixture of 
two incompressible fluids 
considering nonlocal interactions between the molecules. 

Nonlocal Cahn--Hilliard equations 
(see, for instance, 
\cite{ABG-2015, BH-2005, C-1961, DG-2015, GG-2014, 
GLWW-2014, GWW-2014, H-2004}) 
and 
nonlocal Cahn--Hilliard--Navier--Stokes equations 
(see, for instance, 
\cite{CFG-2012, FGG-2016, FG1, FG2, FGK-2013, FRS-2016}) 
have been 
studied by many authors. 
In particular, 
the nonlocal Cahn--Hilliard equation 
(see e.g., \cite{ABG-2015, BH-2005, GG-2014})  
\begin{equation*}
 \begin{cases}\label{E1}\tag{E1}
 \partial_{t}\varphi  - \Delta \mu = 0 & \mbox{in}\ \Omega\times(0, T), \\[-0.5mm]
 \mu = a(\cdot)\varphi -J\ast\varphi + G'(\varphi)  
         & \mbox{in}\ \Omega\times(0, T)
 \end{cases}
\end{equation*}
and the nonlocal Cahn--Hilliard--Navier--Stokes equation 
(see e.g., \cite{CFG-2012, FG1, FG2, FRS-2016}) 
\begin{equation*}
 \begin{cases}\label{E2}\tag{E2}
 \partial_{t}\varphi + u\cdot \nabla\varphi - \Delta \mu = 0 
         & \mbox{in}\ \Omega\times(0, T), \\[-0.5mm]
 \mu = a(\cdot)\varphi -J\ast\varphi + G'(\varphi)  
         & \mbox{in}\ \Omega\times(0, T), \\
 u_{t} - 2\mbox{div}(\nu(\varphi)Du) + (u\cdot\nabla)u + \nabla P 
 = \mu\nabla\varphi + h 
         & \mbox{in}\ \Omega\times(0, T), \\
 \mbox{div}(u) = 0 
         & \mbox{in}\ \Omega\times(0, T)
 \end{cases}
\end{equation*}
have been studied, 
where 
$\Omega \subset \mathbb{R}^{d}$ ($d=2, 3$) is a bounded domain, 
$J$ is an interaction kernel such that $J(x)=J(-x)$ and 
\begin{align*}
(J\ast\varphi)(x):=\int_{\Omega}J(x-y)\varphi(y)\,dy, \qquad 
a(x):=\int_{\Omega}J(x-y)\,dy,
\end{align*} 
$\nu$ denotes the viscosity, $Du:=\frac{1}{2}(\nabla u + (\nabla u)^{tr})$, 
$P$ is the pressure, 
$h$ means volume forces applied to the binary mixture fluid,   
and $G'$ is the first derivative of a double well potential $G$. 
The free energy functional in \eqref{E1} and \eqref{E2} is given by 
$$
E(\varphi(t)):=
\frac{1}{4}\int_{\Omega}\int_{\Omega}J(x-y)(\varphi(x, t)-\varphi(y, t))^2\,dxdy 
+ \int_{\Omega}G(\varphi(x, t))\,dx. 
$$
Existence and energy estimates of weak solutions for 
\eqref{E1} and \eqref{E2}  
have been proved 
by using a Faedo--Galerkin approximation scheme 
considering the compactness of the Neumann operator $-\Delta+1$ 
under the following conditions and others 
(cf.\ \cite{CFG-2012, GG-2014}):  
\begin{enumerate} 
\setlength{\itemsep}{-0.5mm}
 \item[(H1)] $J\in W^{1, 1}(\RN)$, $J(x)=J(-x)$, 
 $a(x):=\int_{\Omega}J(x-y)\,dy \geq0$ a.a.\ $x\in\Omega$. 
 \item[(H2)] $G \in C^{2, 1}_{\rm loc}(\mathbb{R})$ and 
      $$
      G''(s) + \inf_{x\in\Omega}a(x) \geq d_{0}
      $$
 holds for all $s\in\mathbb{R}$ with some constant $d_{0}>0$. 
 \item[(H3)] There exist $d_{1}>\frac{1}{2}\|J\|_{L^1(\RN)}$ 
 and $d_{2} \in \mathbb{R}$ such that 
         $$
         G(s) \geq d_{1}s^2-d_{2}\quad \mbox{for all}\ s\in\mathbb{R}.
         $$
 \item[(H4)] $\varphi_{0}\in L^2(\Omega)$ 
 and 
 $G(\varphi_{0})\in L^1(\Omega)$. 
 \end{enumerate}
 We can verify that 
 the regular potential $G(r) = (r^2-1)^2=r^4-2r^2+1$ ($G'(r)=4r^3-4r$) 
 satisfies (H2)-(H4).    
 The $L^2(0, T; H^1(\Omega))$-estimate for $\mu$ can be established 
 by the Poincar\'e--Wirtinger inequality 
 (see e.g., \cite{CFG-2012, DG-2015, FGG-2016, FG1, FG2, GG-2014}). 
 However, in the case that $\Omega \subset \RN$ is an {\it unbounded} domain, 
 the inequality and 
 the above Galerkin method considering the compactness 
 cannot be used (directly). 
 Moreover, 
 the above regular potential does not satisfy (H4) and  
 the condition (H3) is not appropriate 
 in the case of {\it unbounded} domains 
 because the constant $d_{2}$ is not integrable on {\it unbounded} domains. 

Cahn--Hilliard equations on unbounded domains 
were studied by a few authors (see e.g., \cite{B-2006, FKY-2017, KY1, KY2}).  
In particular, nonlocal Cahn--Hilliard equations on unbounded domains 
have not been studied yet. 
The case of unbounded domains has the mathematical difficult point that 
compactness methods cannot be applied directly. 
It would be interesting to construct an applicable theory for the case of 
unbounded domains and to set assumptions for the case of unbounded domains 
trying to keep a typical example 
in the case of bounded domains (previous works) 
as much as possible. 
By considering the case of unbounded domains, 
it would be possible to make a new finding 
which the case of bounded domains 
does not have. 
Also, the new finding would be useful 
for other study of partial differential equations. 
This article considers the initial-boundary value problem 
on an {\it unbounded} domain for nonlocal Cahn--Hilliard equations  
%
%
 \begin{equation}\label{P}\tag{P}
  \begin{cases} 
    \partial_t\varphi + (-\Delta+1)\mu = 0
         & \mbox{in}\ \Omega\times(0, T), \\[-0.5mm]
    \mu = a(\cdot)\varphi -J\ast\varphi + G'(\varphi)  
         & \mbox{in}\ \Omega\times(0, T), \\[-0mm]
    \partial_{\nu}\mu = 0                                   
         & \mbox{on}\ \partial\Omega\times(0, T),\\[-1mm]
    \varphi(\cdot, 0)=\varphi_0                                    
         & \mbox{in}\ \Omega  
  \end{cases}
\end{equation}
by passing to the limit in the following system as $\ep\searrow0$: 
%
%
%
\begin{equation}\label{Pep}\tag*{(P)$_{\ep}$}
  \begin{cases} 
    \partial_t\varphi_{\ep}+(-\Delta+1)\mu_{\ep} = 0
         & \mbox{in}\ \Omega\times(0, T), \\[-0.5mm]
    \mu_{\ep} =  \ep(-\Delta+1)\varphi_{\ep} 
                      + a(\cdot)\varphi_{\ep} -J\ast\varphi_{\ep} + G_{\ep}'(\varphi_{\ep}) 
                             + \ep\partial_t\varphi_{\ep}
         & \mbox{in}\ \Omega\times(0, T), \\[-0mm]
    \partial_{\nu}\mu_{\ep} = \partial_{\nu}\varphi_{\ep} = 0                                   
         & \mbox{on}\ \partial\Omega\times(0, T),\\[-1mm]
   \varphi_{\ep}(\cdot, 0)=\varphi_{0\ep}                                         
         & \mbox{in}\ \Omega, 
  \end{cases}
\end{equation}
where $\Omega$ is an {\it unbounded} domain in $\RN$ 
with smooth bounded boundary $\partial\Omega$ 
(e.g., $\Omega = \mathbb{R}^{N}\setminus \overline{B(0, R)}$,  
where $B(0, R)$ is the open ball with center $0$ and radius $R>0$), 
$\partial_{\nu}$ denotes differentiation with
respect to the outward normal of $\partial\Omega$, 
$N\in\mathbb{N}$, $T>0$, $\ep>0$, and 
$a(\cdot), J, G, G_{\ep}, \varphi_{0}, \varphi_{0\ep}$ 
are given functions in the following conditions (A1)-(A8):   
%
%
%
 \begin{enumerate} 
 \item[(A1)] $J\in W^{1, 1}(\RN)$, $J(x)=J(-x)$, 
 $a(x):=\int_{\Omega}J(x-y)\,dy \geq0$ a.a.\ $x\in\Omega$. 
 \item[(A2)] $G=\widehat{\beta}+\widehat{\pi}$, 
 where $\widehat{\beta}, \widehat{\pi} \in C^1(\mathbb{R})$.  
 \item[(A3)] $\beta:=\widehat{\beta}\,': \mathbb{R} \to \mathbb{R}$                                
 is a maximal monotone function 
 and $\beta(0) = 0$. 
 $\widehat{\beta}$ is nonnegative and convex and  
 $\widehat{\beta}(0) = 0$. 
 \item[(A4)] $\pi:=\widehat{\pi}' : \mathbb{R} \to \mathbb{R}$ is                          
 a Lipschitz continuous function and $\pi(0) = \widehat{\pi}(0) = 0$.  
 \item[(A5)] $G(r) + \frac{\|\pi'\|_{L^{\infty}(\mathbb{R})}}{2}r^2 \geq 0$ 
 for all $r \in \mathbb{R}$.
 \item[(A6)] $G_{\ep}=\widehat{\beta_{\ep}}+\widehat{\pi}$, 
 where $\widehat{\beta_{\ep}}: \mathbb{R} \to \mathbb{R}$ 
 is the Moreau--Yosida regularization of $\widehat{\beta}$: 
 $$
 \widehat{\beta_{\ep}}(r) 
 = \inf_{s\in\mathbb{R}}\left\{\frac{1}{2\ep}|r-s|^2 + \widehat{\beta}(s) \right\}. 
 $$ 
 \item[(A7)] There exist $c_{0}>0$ and $0<c_{1}<\frac{c_{0}}{2}$ such that 
 $$
 \|J\|_{L^1(\RN)} < c_{0} + c_{1}\quad \mbox{and} \quad 
 \inf_{x\in\Omega} a(x) \geq c_{0} + \|\pi'\|_{L^{\infty}(\mathbb{R})}. 
 $$ 
 \item[(A8)] $\varphi_{0}\in L^2(\Omega)$ 
 and 
 $G(\varphi_{0})\in L^1(\Omega)$. Moreover, 
 let $\varphi_{0\ep}\in H^1(\Omega)$ 
 satisfy $G(\varphi_{0\ep})\in L^1(\Omega)$ 
 and 
   \begin{equation*}
   \|\varphi_{0\ep}\|_{L^2(\Omega)}^2 \leq c_{2},\quad 
   \|G(\varphi_{0\ep})\|_{L^1(\Omega)} \leq c_{2}, \quad 
   \ep\|\nabla \varphi_{0\ep}\|_{(L^2(\Omega))^N}^2 \leq c_{2}
   \end{equation*}
 for all $\ep>0$, where $c_{2} > 0$ 
 is a constant independent of $\ep$; 
 in addition, 
 $\varphi_{0\ep}\to \varphi_{0}$ in $L^2(\Omega)$ as 
 $\ep\searrow0$.
 \end{enumerate}
The function $G(r) = r^4-2r^2$ ($G'(r)=4r^3-4r$) satisfies (A1)-(A8) 
(see Section \ref{Sec2}). 
%
%
%
\begin{remark}\label{MYreg}
It holds that 
$\widehat{\beta_{\ep}}(r)=\frac{1}{2\ep}|r-J_{\ep}^{\beta}(r)|^2 
+ \widehat{\beta}(J_{\ep}^{\beta}(r))$ 
for all $r\in\mathbb{R}$, 
where $J_{\ep}^{\beta}$ is the resolvent operator of $\beta$ 
on $\mathbb{R}$.  
The derivative of $\widehat{\beta_{\ep}}$ is $\beta_{\ep}$, 
where $\beta_{\ep}$ is the Yosida approximation operator of $\beta$ 
on $\mathbb{R}$.   
Moreover, the inequalities $0\leq \widehat{\beta_{\ep}}(r) \leq \widehat{\beta}(r)$ 
hold for all $r\in\mathbb{R}$ (see e.g., \cite[Theorem 2.9, p.\ 48]{Barbu2}). 
\end{remark}

%
%
%
This article puts the Hilbert spaces 
   $$
   H:=L^2(\Omega), \quad V:=H^1(\Omega)
   $$
 with inner products $(u_{1}, u_{2})_{H}:=\int_{\Omega}u_{1}u_{2}\,dx$ 
 ($u_{1}, u_{2} \in H$)  
 and $(v_{1}, v_{2})_{V}:=
 \int_{\Omega}\nabla v_{1}\cdot\nabla v_{2}\,dx + \int_{\Omega} v_{1}v_{2}\,dx$ 
 ($v_{1}, v_{2} \in V$), 
 respectively, 
 and with norms $\|u\|_{H}:=(u, u)_{H}^{1/2}$ ($u\in H$) and 
 $\|v\|_{V}:=(v, v)_{V}^{1/2}$ ($v\in V$), respectively.  
 Moreover, this paper uses  
   $$
   W:=\bigl\{z\in H^2(\Omega)\ |\ \partial_{\nu}z = 0 \quad 
   \mbox{a.e.\ on}\ \partial\Omega\bigr\}.
   $$
 The notation $V^{*}$ denotes the dual space of $V$ with 
 duality pairing $\langle\cdot, \cdot\rangle_{V^*, V}$. 
 Moreover, in this paper, a bijective mapping $F : V \to V^{*}$ and 
 the inner product in $V^{*}$ are defined as 
    \begin{align}
    &\langle Fv_{1}, v_{2} \rangle_{V^*, V} := 
    (v_{1}, v_{2})_{V} \quad \mbox{for all}\ v_{1}, v_{2}\in V, 
    \label{defF}
    \\[0mm]
    &(v_{1}^{*}, v_{2}^{*})_{V^{*}} := 
    \left\langle v_{1}^{*}, F^{-1}v_{2}^{*} 
    \right\rangle_{V^*, V} 
    \quad \mbox{for all}\ v_{1}^{*}, v_{2}^{*}\in V^{*};
    \label{innerVstar}
    \end{align}
 note that $F : V \to V^{*}$ is well-defined by 
 the Riesz representation theorem.  

This article defines weak solutions of \eqref{P} and \ref{Pep} as follows. 
%
%
%
\begin{df}        
 A pair $(\varphi, \mu)$ with 
    \begin{align*}
    &\varphi \in H^1(0, T; V^{*})\cap L^{\infty}(0, T; H)\cap L^2(0, T; V), 
    \\
    &\mu\in L^2(0, T; V)
    \end{align*}
 is called a {\it weak solution} of \eqref{P} if 
 $(\varphi, \mu)$ 
 satisfies 
    \begin{align}
        &\langle \varphi_{t}(t), v \rangle_{V^{*}, V} 
           + \bigl(\mu(t), v \bigr)_{V} 
          = 0 
          \quad 
          \mbox{for all}\ v \in V\ 
          \mbox{and a.a}.\ t\in(0, T), \label{defsol1}
     \\[0mm]
        &\mu = a(\cdot)\varphi -J\ast\varphi + G'(\varphi)                           
        \quad \mbox{a.e.\ on}\ \Omega\times(0, T), \label{defsol2}
     \\[0mm]
        &\varphi(0)=\varphi_{0} 
        \quad \mbox{a.e.\ on}\ \Omega. \label{defsol3}
     \end{align}
\end{df}

%
%
%
%
\begin{df}        
 A pair $(\varphi_{\ep}, \mu_{\ep})$ with 
    \begin{align*}
    &\varphi_{\ep} \in H^1(0, T; H)\cap L^{\infty}(0, T; V)\cap L^2(0, T; W), 
    \\
    &\mu_{\ep}\in L^2(0, T; V)
    \end{align*}
 is called a {\it weak solution} of \ref{Pep} if 
 $(\varphi_{\ep}, \mu_{\ep})$ 
 satisfies 
    \begin{align}
        &(\partial_t\varphi_{\ep}(t), v)_{H} 
           + (\mu_{\ep}(t), v)_{V} 
          = 0 
          \quad 
          \mbox{for all}\ v \in V\ 
          \mbox{and a.a}.\ t\in(0, T), \label{defsolep1}
     \\[0mm]
        &\mu_{\ep} =  \ep(-\Delta+1)\varphi_{\ep} 
                      + a(\cdot)\varphi_{\ep} -J\ast\varphi_{\ep} + G_{\ep}'(\varphi_{\ep}) 
                             + \ep\partial_t\varphi_{\ep} 
        \quad \mbox{a.e.\ on}\ \Omega\times(0, T), \label{defsolep2}
     \\[0mm]
        &\varphi_{\ep}(0)=\varphi_{0\ep} 
        \quad \mbox{a.e.\ on}\ \Omega. \label{defsolep3}
     \end{align}
\end{df}

This paper has four main theorems. 
The first main result 
gives existence and uniqueness of solutions to \ref{Pep}.  
\begin{thm}\label{maintheorem1}
 Assume {\rm (A1)-(A8)}. 
 Then there exists $\ep_{0}\in(0, 1)$ such that for all $\ep\in(0, \ep_{0})$  
 there exists a unique weak solution $(\varphi_{\ep}, \mu_{\ep})$ of {\rm \ref{Pep}} 
 satisfying                                                 
     \begin{equation*}
        \varphi_{\ep} \in H^1(0, T; H)\cap L^{\infty}(0, T; V)\cap L^2(0, T; W),
         \quad \mu_{\ep} \in L^2(0, T; V) 
     \end{equation*}
 and there exists a constant $M_{1}=M_{1}(T)>0$ 
 such that 
     \begin{align}
     &\|\varphi_{\ep}(t)\|_{H}^2 + \ep\|\varphi_{\ep}(t)\|^2_{V} 
       + \int_{0}^{t}\|\mu_{\ep}(s)\|^2_{V}\,ds  \label{essolep1} \\[-0.5mm] 
      &+ \ep\int_{0}^{t}\|(-\Delta+1)\varphi_{\ep}(s)\|^2_{H}\,ds 
       + \ep\int_{0}^{t}\|\partial_{t}\varphi_{\ep}(s)\|^2_{H}\,ds 
     \leq M_{1},  \notag
     \\[1mm]
     &\int_{0}^{t}\|\varphi_{\ep}(s)\|^2_{V}\,ds 
     \leq M_{1}, \label{essolep2} 
     \\
     &\int_{0}^{t}\|\partial_{t}\varphi_{\ep}(s)\|^2_{V^{*}}\,ds \leq M_{1}, 
     \label{essolep3} 
     \\
     &\int_{0}^{t}\|\beta_{\ep}(\varphi_{\ep}(s))\|_{H}^2\,ds 
     \leq M_{1} \label{essolep4}
     \end{align}
 for all $t\in[0, T]$ and all $\ep\in(0, \ep_{0})$.
 \end{thm}

The second main result says existence and uniqueness 
of solutions to \eqref{P}.   
%
%
\begin{thm}\label{maintheorem2}
 Assume {\rm (A1)-(A8)}. 
 Then there exists a unique weak solution $(\varphi, \mu)$ of {\rm \eqref{P}} 
 satisfying                                                 
     \begin{equation*}
        \varphi \in H^1(0, T; V^{*})\cap L^{\infty}(0, T; H)\cap L^2(0, T; V),
         \quad \mu \in L^2(0, T; V) 
     \end{equation*}
 and there exists a constant $M_{2}=M_{2}(T)>0$ 
 such that 
     \begin{align}
     &\|\varphi(t)\|_{H}^2 
       + \int_{0}^{t}\|\mu(s)\|^2_{V}\,ds  
     \leq M_{2},  \label{es1} 
     \\[0mm]
     &\int_{0}^{t}\|\varphi(s)\|^2_{V}\,ds 
     \leq M_{2}, \label{es2} 
     \\
     &\int_{0}^{t}\|\varphi_{t}(s)\|^2_{V^{*}}\,ds \leq M_{2}, \label{es3} 
     \\
     &\int_{0}^{t}\|\beta(\varphi(s))\|_{V}^2\,ds 
     \leq M_{2} \label{es4}
     \end{align}
 for all $t\in[0, T]$.
 \end{thm}

The third main result 
is concerned with the energy estimate for \eqref{P}.  
%
%
\begin{thm}\label{maintheorem3} 
Assume {\rm (A1)-(A8)}. 
Let $(\varphi, \mu)$ be a weak solution of {\rm \eqref{P}}. 
Then   
\begin{align*}
E(\varphi(t)) + \int_{0}^{t}\|\mu(s)\|_{V}^2\,ds = E(\varphi_{0}), 
\hspace{1em} \mbox{in particular,} 
\hspace{1em} E(\varphi(t)) \leq E(\varphi_{0})
\end{align*}
for all $t\in[0, T]$, 
where 
$$
E(\varphi(t)):=
\frac{1}{4}\int_{\Omega}\int_{\Omega}J(x-y)(\varphi(x, t)-\varphi(y, t))^2\,dxdy 
+ \int_{\Omega}G(\varphi(x, t))\,dx. 
$$
\end{thm}

The fourth main result infers the error estimate 
between the solution of \eqref{P} and the solution of \ref{Pep}.  
%
%
\begin{thm}\label{maintheorem4}
Assume {\rm (A1)-(A8)}. 
In {\rm (A8)} assume further that 
$$
\|\varphi_{0\ep}-\varphi_0\|_{V^{*}}^2 \leq c_3\ep^{1/2}
$$
for some constant $c_3>0$. Let $\ep_{0}$ be as in Theorem \ref{maintheorem1}. 
For $\ep \in (0, \ep_{0})$, let $(\varphi_{\ep}, \mu_{\ep})$ and $(\varphi, \mu)$ 
be weak solutions of {\rm \ref{Pep}} and {\rm \eqref{P}}, respectively. 
Then there exists a constant $M_3=M_3(T)>0$ such that 
\begin{align*}
\|\varphi_{\ep}-\varphi\|_{C([0, T]; V^{*})}^2 
+ \int_{0}^{T} \|\varphi_{\ep}(t)-\varphi(t)\|_{H}^2\,dt 
\leq M_3\ep^{1/2}
\end{align*}
for all $\ep\in(0, \ep_{0})$.
\end{thm}

The strategy for the proof of Theorem \ref{maintheorem1} is to consider 
the approximation of \ref{Peplam}
\begin{equation}\label{Peplam}\tag*{(P)$_{\ep, \lambda}$}
  \begin{cases} 
    \lambda\partial_t\mu_{\ep, \lambda}
    + \partial_t\varphi_{\ep, \lambda}+((-\Delta)_{\lambda}+1)\mu_{\ep, \lambda} = 0
         & \mbox{in}\ \Omega\times(0, T), \\[-0.5mm]
    \mu_{\ep, \lambda} =  \ep((-\Delta)_{\lambda}+1)\varphi_{\ep, \lambda} 
                      + a(\cdot)\varphi_{\ep, \lambda} 
                      -J\ast\varphi_{\ep, \lambda}\\ 
                                    \hspace{62.7mm} +G_{\ep}'(\varphi_{\ep, \lambda}) 
                             + \ep\partial_t\varphi_{\ep, \lambda}
         & \mbox{in}\ \Omega\times(0, T), \\[-0mm]
   \mu_{\ep, \lambda}(\cdot, 0)=\varphi_{0\ep},\  
   \varphi_{\ep, \lambda}(\cdot, 0)=\varphi_{0\ep}                                         
         & \mbox{in}\ \Omega, 
  \end{cases}
\end{equation}
where $\lambda>0$ 
and $(-\Delta)_{\lambda}$ is the Yosida approximation of $-\Delta$, 
to establish existence and estimates of solutions for \ref{Peplam} 
(Lemmas \ref{existencePeplam} and \ref{estimatePeplam}), 
and to pass to the limit in \ref{Peplam} as $\lambda\searrow0$. 
The strategy for the proof of Theorem \ref{maintheorem2} 
is to confirm Cauchy's criterion for solutions of \ref{Pep} (Lemma \ref{Cauchy}) 
and to pass to the limit in \ref{Pep} as $\ep\searrow0$. 

This paper is organized as follows. 
Section \ref{Sec2} presents one example. 
In Section \ref{Sec3} we give useful results for 
proving the main theorems. 
Sections \ref{Sec4} and \ref{Sec5} are devoted to the proofs of 
Theorems \ref{maintheorem1} and \ref{maintheorem2}. 
In Section \ref{Sec6} we establish the energy estimate for \eqref{P} 
stated in Theorem \ref{maintheorem3}. 
Section \ref{Sec7} proves the error estimate 
between the solution of \eqref{P} and the solution of \ref{Pep} 
stated in Theorem \ref{maintheorem4}.

 \section{Example}\label{Sec2}

This paper presents the example: 
\begin{align*}
&\Omega=\RN\setminus B(0, \eta) 
\quad \mbox{($\eta>0$ is sufficiently small)}, 
\\ 
&J(x)=c_{J}e^{-|x|^2} \quad 
\mbox{($c_{J}>0$ is a constant)}, 
\\ 
&G(r)= r^4-2r^2. 
\end{align*}
By letting $c_{J}$ be a positive constant such that $\|J\|_{L^1(\RN)} = 21$ 
and 
putting 
\begin{align*}
&a(x):=\int_{\Omega}J(x-y)\,dy,\  
\widehat{\beta}(r)=r^4,\ \widehat{\pi}(r)=-2r^2, 
\\   
&G_{\ep}(r) = \widehat{\beta_{\ep}}(r) + \widehat{\pi}(r),\  
c_{0}=16,\ 0<c_{1}=6<\frac{c_{0}}{2}, 
\end{align*}
where $\widehat{\beta_{\ep}}$ is the Moreau--Yosida regularization 
of $\widehat{\beta}$, 
these functions $a$, $J$, $G$ and $G_{\ep}$ satisfy (A1)-(A8). 
Indeed, 
\begin{align*}
\|J\|_{L^1(\RN)} = 21 < 16 + 6 = c_{0} + c_{1}
\end{align*}
and 
\begin{align*}
a(x)&=\int_{\Omega}J(x-y)\,dy = \|J\|_{L^1(\RN)} - \int_{B(0, \eta)}J(x-y)\,dy
=21-\int_{x-B(0, \eta)}J(z)\,dz 
\\
&\geq 21 - c_{J}|x-B(0, \eta)| = 21 - c_{J}|B(0, \eta)| 
\\
&\geq 20 = c_{0}+\|\pi'\|_{L^{\infty}(\mathbb{R})}
\end{align*}
hold, which implies (A7). 

It is possible to verify (A8) in reference to \cite[Section 6]{KY2}. 
To confirm (A8) we let  
 $\varphi_{0} \in L^2(\Omega)$ with 
 $G(\varphi_{0}) \in L^1(\Omega)$, 
 i.e., 
 $\varphi_{0} \in L^2(\Omega)\cap L^{4}(\Omega)$. 
 Then there exists $\varphi_{0\ep} \in W\cap Y$ such that 
    \begin{equation*}
       \begin{cases}
         \varphi_{0\ep} + \ep(-\Delta + 1)\varphi_{0\ep} = \varphi_{0}
         \quad \mbox{in}\ \Omega,
         \\[2mm]
         \partial_{\nu}\varphi_{0\ep} = 0 
         \quad \mbox{on}\ \partial\Omega, 
       \end{cases}
    \end{equation*}
 that is, 
 $$\varphi_{0\ep} = (J_{H})_{\ep}\varphi_{0} = (J_{L^{4}})_{\ep}\varphi_{0},$$ 
 where  
 \begin{align*}
 &A_{H}:=-\Delta+I 
 : D(A):=W \subset H \to H, 
 \\[1.5mm]
 &(J_{H})_{\ep}:=(I+\ep A_{H})^{-1}, 
 \\[1.5mm]
 &Y:=\bigl\{z \in W^{2,\,4}(\Omega)\ |\ \partial_{\nu}z = 0 
 \quad \mbox{a.e.\ on}\ \partial\Omega \bigr\}, 
 \\[1.5mm]
 &A_{L^{4}}:=-\Delta+I 
 : D(A_{L^{4}}):=Y \subset L^{4}(\Omega) \to L^{4}(\Omega), 
 \\[1.5mm]
 &(J_{L^{4}})_{\ep}:=(I+\ep A_{L^{4}})^{-1}.  
 \end{align*}
 It follows from the properties of 
 $(J_{H})_{\ep}$ and $(J_{L^{4}})_{\ep}$ that 
  \begin{align*}
  &\varphi_{0\ep} = (J_{H})_{\ep}\varphi_{0} \to \varphi_{0} 
  \quad \mbox{in}\ H\ \mbox{as}\ \ep \searrow 0, 
  \\[1.5mm]
  &\|\varphi_{0\ep}\|_{H} 
                      = \|(J_{H})_{\ep}\varphi_{0}\|_{H} 
  \leq \|\varphi_{0}\|_{H}, 
  \\[1.5mm]
  &\|\varphi_{0\ep}\|_{L^{4}(\Omega)} 
                = \|(J_{L^{4}})_{\ep}\varphi_{0}\|_{L^{4}(\Omega)} 
  \leq \|\varphi_{0}\|_{L^{4}(\Omega)}, 
  \end{align*}
  and hence 
  \begin{align}
  &\|G(\varphi_{0\ep})\|_{L^1(\Omega)} \notag
  \leq \|\varphi_{0\ep}\|_{L^{4}(\Omega)}^{4} + 2\|\varphi_{0\ep}\|_{H}^2
  \leq \|\varphi_{0}\|_{L^{4}(\Omega)}^{4} + 2\|\varphi_{0}\|_{H}^2, 
  \\[2mm]
  &\ep\|\varphi_{0\ep}\|_{V}^2     \label{6.1}
  = \bigl(\ep(-\Delta+I)\varphi_{0\ep}, \varphi_{0\ep}\bigr)_{H} 
  = (\varphi_{0}-\varphi_{0\ep}, \varphi_{0\ep})_{H} 
  \leq \|\varphi_{0}\|_{H}^2. 
  \end{align}
 Thus there exists $\varphi_{0\ep}$ satisfying (A8). 
 
 Moreover, the inequality  
 $$
 \|\varphi_{0\ep}-\varphi_{0}\|_{V^{*}} \leq \ep^{1/2}\|\varphi_{0}\|_{H}
 $$ 
 holds. 
 Indeed, \eqref{defF}, \eqref{innerVstar} and \eqref{6.1} yield that 
 $$
 \|\varphi_{0\ep}-\varphi_{0}\|_{V^{*}}^2 
 = \|\ep(-\Delta+1)\varphi_{0\ep}\|_{V^{*}}^2 
 = \ep^2\|F\varphi_{0\ep}\|_{V^{*}}^2 
 = \ep^2\|\varphi_{0\ep}\|_{V}^2 
 \leq \ep\|\varphi_{0}\|_{H}^2.
 $$

Therefore (A1)-(A8) hold 
for the functions $a$, $J$, $G$ and $G_{\ep}$ in the example. 

\section{Preliminaries}\label{Sec3}

In this section we will provide some results which will be used later for 
the proofs of Theorems \ref{maintheorem1} and \ref{maintheorem2}. 
 \begin{lem}[{\cite[Section 8, Corollary 4]{Si-1987}}] \label{Ascoli}                                                                   
 Assume that 
 $$
 X \subset Z \subset Y \ \mbox{with compact embedding}\ X \to Z\ 
 \mbox{$($$X$, $Z$ and $Y$ are Banach spaces$)$.}
 $$
 \begin{enumerate}
 \item[(i)] Let $F$ be bounded in $L^p(0, T; X)$ and 
 $\{\frac{\partial v}{\partial t}\ |\ v\in F \}$ 
 be bounded in $L^1(0, T; Y)$ with some constant $1\leq p<\infty$. 
 Then $F$ is relatively compact in $L^p(0, T; Z)$. 
 \item[(ii)] Let $F$ be bounded in $L^{\infty}(0, T; X)$ and 
 $\{\frac{\partial v}{\partial t}\ |\ v\in F \}$ 
 be bounded in $L^r(0, T; Y)$ with some constant $r>1$. 
 Then $F$ is relatively compact in $C(0, T; Z)$. 
 \end{enumerate}
 \end{lem}
%
 \begin{lem}\label{ra}
  Let $\lambda>0$ and put 
 \begin{align*}
 &J_{\lambda}:=(I - \lambda \Delta)^{-1} : H \to H, \quad
    (-\Delta)_{\lambda}:= \frac{1}{\lambda}(I - J_{\lambda}) : H \to H,  
 \\ 
 &\tilde{A}:=F-I : V \to V^{*}, \quad 
 \tilde{J}_{\lambda}:=
 \bigl(I + \lambda\tilde{A}\bigr)^{-1} : V^{*} \to V^{*}. 
 \end{align*}
 Then we have 
 \begin{align}
 &\tilde{J}_{\lambda}|_{H} = J_{\lambda}, \label{tool1}\\
 &\|J^{1/2}_{\lambda}v\|_{H} \leq \|v\|_{H}, \label{tool2}\\
 &\|J^{1/2}_1v\|_{V} = \|v\|_{H} \label{tool3} 
 \end{align}
 for all $v \in H$ 
 and 
 \begin{align}\label{tool4}
 \|(-\Delta)^{1/2}_{\lambda} v\|_{H} \leq \|v\|_{V} 
 \end{align}
 for all $v \in V$, 
 where 
 $-\Delta : W \subset H \to H$ is the Neumann Laplacian. 
 \end{lem}
\begin{proof}
\eqref{tool1} can be shown by the same argument as in \cite[Lemma 3.3]{KY2}. 
We will prove \eqref{tool2}. From the properties of $J_{\lambda}$ we have that 
$$
\|J^{1/2}_{\lambda}v\|^2_{H} = (J^{1/2}_{\lambda}v, J^{1/2}_{\lambda}v)_{H} 
                                   = (v, J_{\lambda}v)_{H} 
                                  \leq \|v\|_{H}\|J_{\lambda}v\|_{H}
                                  \leq \|v\|^2_{H}
$$
for all $v \in H$. 
Thus \eqref{tool2} holds. 

Next we show \eqref{tool3}. 
By noting that $F=\tilde{A} + I$ and 
$(\tilde{J}^{1/2}_1 v^{*}, v)_{H} = \langle v^{*}, J^{1/2}_1 v \rangle_{V^{*}, V}$ 
for all $v^{*} \in V^{*}$ and all $v \in H$ (see e.g., \cite[Lemma 3.3]{OSY}), 
it follows from \eqref{defF} and \eqref{tool1} that 
\begin{align*}
\|J^{1/2}_1 v\|^2_{V} 
            &= (J^{1/2}_1 v, J^{1/2}_1 v)_{V} 
            = \langle FJ^{1/2}_1 v, J^{1/2}_1 v \rangle_{V^*, V} 
            = \langle F\tilde{J}^{1/2}_1 v, J^{1/2}_1 v \rangle_{V^*, V} 
            \\
            &= (\tilde{J}^{1/2}_1 F\tilde{J}^{1/2}_1 v, v)_{H} 
            = ((I+\tilde{A})^{1/2}\tilde{J}^{1/2}_1 v, v)_{H} 
            = (v, v)_{H} 
            = \|v\|^2_{H}
\end{align*}
for all $v \in H$, which implies \eqref{tool3}. 

Next we confirm \eqref{tool4}. Let $v \in V$. 
Then it holds that  
\begin{align}\label{tool4-1}
\|(-\Delta)_{\lambda}^{1/2}v\|_{H}^2 
&= ((-\Delta)_{\lambda}v, v)_{H} 
= (-\Delta J_{\lambda}v, v)_{H} 
= \langle (F-I)J_{\lambda}v, v \rangle_{V^{*}, V}
\\ \notag
&\leq \frac{1}{2}\|(F-I)J_{\lambda}v\|_{V^{*}}^2 + \frac{1}{2}\|v\|_{V}^2. 
\end{align}
Here we infer from \eqref{defF} and \eqref{innerVstar} that 
\begin{align}\label{tool4-2}
&\|(F-I)J_{\lambda}v\|_{V^{*}}^2 
\\ \notag
&= \|FJ_{\lambda}v\|_{V^{*}}^2 -2(FJ_{\lambda}v, J_{\lambda}v)_{V^{*}} 
   + \|J_{\lambda}v\|_{V^{*}}^2
=\|J_{\lambda}v\|_{V}^2 - 2\|J_{\lambda}v\|_{H}^2 + \|J_{\lambda}v\|_{V^{*}}^2 
\\ \notag 
&\leq \|J_{\lambda}v\|_{V}^2 - \|J_{\lambda}v\|_{H}^2 
= (J_{\lambda}v, (-\Delta)_{\lambda}v)_{H} 
= - \lambda \|(-\Delta)_{\lambda}v\|_{H}^2 + (v, (-\Delta)_{\lambda}v)_{H} 
\\ \notag
&\leq (v, (-\Delta)_{\lambda}v)_{H} 
=\|(-\Delta)_{\lambda}^{1/2}v\|_{H}^2. 
\end{align}
Hence combination of \eqref{tool4-1} and \eqref{tool4-2} derives \eqref{tool4}.

\end{proof} 
%
\begin{lem}\label{keylemma}
Let $\widetilde{\Omega} \subset \Omega$ be 
a bounded domain with smooth boundary 
and 
let $\{v_{\lambda}\}_{\lambda} \subset H^1(0, T; H) \cap L^{\infty}(0, T; H)$ 
satisfy that 
$\{v_{\lambda}\}_{\lambda}$ is bounded in $L^{\infty}(0, T; H)$ and 
$\{v_{\lambda}'\}_{\lambda}$ is bounded in $L^2(0, T; H)$. 
Then 
\begin{align*}
&v_{\lambda} \to v\ \mbox{weakly$^{*}$ in}\ L^{\infty}(0, T; H), \\
&J^{1/2}_{1}v_{\lambda} \to J^{1/2}_{1} v\ 
\mbox{in}\ C([0, T]; L^2(\widetilde{\Omega}))
\end{align*}
as $\lambda = \lambda_j \searrow0$ 
with some function $v \in L^{\infty}(0, T; H)$. 
\end{lem}
\begin{proof}
There exists $v \in  L^{\infty}(0, T; H)$ such that 
\begin{align*}
&v_{\lambda} \to v\ 
\mbox{weakly$^{*}$ in}\ L^{\infty}(0, T; H)
\end{align*}
as $\lambda = \lambda_j \searrow0$. 
We see that 
\begin{equation}\label{embedding}
H^1(\widetilde{\Omega}) \subset L^2(\widetilde{\Omega}) 
\subset L^2(\widetilde{\Omega})\ 
\mbox{with compact embedding}\ 
H^1(\widetilde{\Omega}) \to L^2(\widetilde{\Omega}).  
\end{equation}
It follows from \eqref{tool3} that 
$$
\|J^{1/2}_1 v_{\lambda}(t)\|_{H^1(\widetilde{\Omega})} 
\leq \|J^{1/2}_1v_{\lambda}(t)\|_{V} 
= \|v_{\lambda}(t)\|_{H}.  
$$
Thus  
there exists a constant $C_1>0$ such that 
\begin{equation}\label{Linfty}
\|J^{1/2}_1 v_{\lambda} \|_{L^{\infty}(0, T; H^1(\widetilde{\Omega}))} \leq C_1. 
\end{equation}
Also, from \eqref{tool2} we have that 
$$
\|J^{1/2}_1\partial_{t}v_{\lambda}(t)\|_{L^2(\widetilde{\Omega})}
\leq \|J^{1/2}_1 \partial_{t}v_{\lambda}(t)\|_{H}
\leq \|\partial_{t}v_{\lambda}(t)\|_{H},
$$
and hence it holds that there exists a constant $C_2>0$ such that 
\begin{equation}\label{L2}
\|J^{1/2}_1 \partial_{t}v_{\lambda}\|_{L^2(0, T; L^2(\widetilde{\Omega}))} \leq C_2.
\end{equation}
Therefore 
applying \eqref{embedding}-\eqref{L2} and Lemma \ref{Ascoli} 
yields that 
\begin{equation}\label{AfterAscoli}
J^{1/2}_1 v_{\lambda} \to w \quad \mbox{in}\ C([0, T]; L^2(\widetilde{\Omega}))
\end{equation}
as $\lambda = \lambda_j \searrow0$ 
with some function $w \in C([0, T]; L^2(\widetilde{\Omega}))$. 
Now, let $\psi \in C_{\mathrm c}^{\infty}([0, T] \times \widetilde{\Omega})$ 
and we will show that 
\begin{equation}\label{henbunhou}
\int_0^T \Bigl(\int_{\widetilde{\Omega}} \bigl(J^{1/2}_1v(t) -w(t) \bigr)\psi(t) 
                                                                                                 \Bigr)\,dt 
= 0.
\end{equation} 
We see that 
\begin{align}\label{LL}
\int_0^T \Bigl(\int_{\widetilde{\Omega}} 
                                      \bigl(J^{1/2}_1v_{\lambda}(t) \bigr)\psi(t) \Bigr)\,dt
= \int_0^T (J^{1/2}_1v_{\lambda}(t), \psi(t))_{H}\,dt
= \int_0^T (v_{\lambda}(t), J^{1/2}_1\psi(t))_{H}\,dt. 
\end{align}
Where, 
since 
$\psi \in C_{\mathrm c}^{\infty}([0, T]\times \widetilde{\Omega}) 
\subset C_{\mathrm c}^{\infty}([0, T]\times \Omega) 
\subset L^1(0, T; H)$,   
we infer from \eqref{tool2} that 
$$
J^{1/2}_1\psi \in L^1(0, T; H).
$$
Therefore it follows that 
\begin{align}\label{kyokugen}
\int_0^T (v_{\lambda}(t), J^{1/2}_1\psi(t))_{H}\,dt 
\to \int_0^T (v(t), J^{1/2}_1\psi(t))_{H}\,dt
= \int_0^T (J^{1/2}_1v(t), \psi(t))_{H}\,dt
\end{align}
as $\lambda = \lambda_j \searrow0$. 
Thus combination of \eqref{AfterAscoli}, \eqref{LL} and \eqref{kyokugen} 
leads to \eqref{henbunhou}, and hence it holds that 
\begin{equation}\label{Afterhenbunhou}
w = J^{1/2}_1v \quad \mbox{a.e.\ in}\ (0, T) \times \widetilde{\Omega}.
\end{equation} 
From \eqref{AfterAscoli} and \eqref{Afterhenbunhou} we have 
\begin{equation}\label{itiyou}
J^{1/2}_1v_{\lambda} \to J^{1/2}_1 v 
\quad \mbox{in}\ C([0, T]; L^2(\widetilde{\Omega}))
\end{equation}
as $\lambda = \lambda_j \searrow0$. 
\end{proof}
%
\begin{lem}\label{keylemma2}
Let $\widetilde{\Omega} \subset \Omega$ be 
a bounded domain with smooth boundary 
and 
let $\{v_{\lambda}\}_{\lambda} \subset H^1(0, T; H)$ 
satisfy that 
$\{v_{\lambda}\}_{\lambda}$ is bounded in $L^2(0, T; H)$ and 
$\{v_{\lambda}'\}_{\lambda}$, $\{(-\Delta)_{\lambda}v_{\lambda}\}_{\lambda}$ 
are bounded in $L^2(0, T; H)$. 
Then 
\begin{align*}
&v_{\lambda} \to v\ 
\mbox{in}\ L^2(0, T; L^2(\widetilde{\Omega}))
\end{align*}
as $\lambda = \lambda_j \searrow0$ 
with some function $v \in L^2(0, T; W)$.
\end{lem}
\begin{proof}
There exists $v \in  L^2(0, T; W)$ such that 
\begin{align*}
&v_{\lambda} \to v\ 
\mbox{weakly in}\ L^2(0, T; H), 
\\ 
&(-\Delta)_{\lambda}v_{\lambda} \to -\Delta v\ 
\mbox{weakly in}\ L^2(0, T; H) 
\end{align*}
as $\lambda = \lambda_j \searrow0$. 
We see that 
\begin{align*}
\|J_{\lambda}v_{\lambda}(t)\|_{H^1(\widetilde{\Omega})}^2
&\leq \|J_{\lambda}v_{\lambda}(t)\|_{V}^2 
= \|J_{\lambda}v_{\lambda}(t)\|_{H}^2 
   + (J_{\lambda}v_{\lambda}(t), (-\Delta)_{\lambda}v_{\lambda}(t))_{H} 
\\ 
&\leq \frac{3}{2}\|v_{\lambda}(t)\|_{H}^2 
      + \frac{1}{2}\|(-\Delta)_{\lambda}v_{\lambda}(t)\|_{H}
\end{align*}
and then 
\begin{align}\label{dougu1}
\|J_{\lambda}v_{\lambda}\|_{L^2(0, T; H^1(\widetilde{\Omega}))} 
\leq C_{1}
\end{align}
with some constant $C_{1}>0$. 
It holds that 
\begin{align*}
\|J_{\lambda}\partial_{t}v_{\lambda}(t)\|_{L^2(\widetilde{\Omega})} 
\leq \|J_{\lambda}\partial_{t}v_{\lambda}(t)\|_{H} 
\leq \|\partial_{t}v_{\lambda}(t)\|_{H},  
\end{align*}
and hence there exists a constant $C_{2}>0$ such that 
\begin{align}\label{dougu2}
\|J_{\lambda}\partial_{t}v_{\lambda}\|_{L^2(0, T; L^2(\widetilde{\Omega}))}  
\leq C_{2}. 
\end{align}
Thus it follows from 
\eqref{embedding}, \eqref{dougu1}, \eqref{dougu2} and Lemma \ref{Ascoli} 
that  
\begin{equation*}
J_{\lambda}v_{\lambda} \to v \quad 
\mbox{in}\ L^2(0, T; L^2(\widetilde{\Omega}))
\end{equation*}
as $\lambda = \lambda_j \searrow0$, which implies that 
\begin{equation*}
v_{\lambda} = \lambda(-\Delta)_{\lambda}v_{\lambda} + J_{\lambda}v_{\lambda} 
\to v \quad \mbox{in}\ 
L^2(0, T; L^2(\widetilde{\Omega}))
\end{equation*}
as $\lambda = \lambda_j \searrow0$. 
\end{proof}

\vspace{10pt}

\section{Existence of solutions to \ref{Pep}}\label{Sec4}

To show existence of weak solutions for \ref{Pep} 
this paper considers the approximation of \ref{Pep}: 
%
%
%
\begin{equation}\label{Peplam}\tag*{(P)$_{\ep, \lambda}$}
  \begin{cases} 
    \lambda\partial_t\mu_{\ep, \lambda}
    + \partial_t\varphi_{\ep, \lambda}+((-\Delta)_{\lambda}+1)\mu_{\ep, \lambda} = 0
         & \mbox{in}\ \Omega\times(0, T), \\[-0.5mm]
    \mu_{\ep, \lambda} =  \ep((-\Delta)_{\lambda}+1)\varphi_{\ep, \lambda} 
                      + a(\cdot)\varphi_{\ep, \lambda} 
                      -J\ast\varphi_{\ep, \lambda}\\ 
                                    \hspace{62.7mm} +G_{\ep}'(\varphi_{\ep, \lambda}) 
                             + \ep\partial_t\varphi_{\ep, \lambda}
         & \mbox{in}\ \Omega\times(0, T), \\[-0mm]
   \mu_{\ep, \lambda}(\cdot, 0)=\varphi_{0\ep},\  
   \varphi_{\ep, \lambda}(\cdot, 0)=\varphi_{0\ep}                                         
         & \mbox{in}\ \Omega, 
  \end{cases}
\end{equation}
where $\lambda>0$ 
and $(-\Delta)_{\lambda}$ is the Yosida approximation of $-\Delta$.

\begin{lem}\label{existencePeplam}
There exists 
a unique classical solution $(\varphi_{\ep, \lambda}, \mu_{\ep, \lambda})$ of 
{\rm \ref{Peplam}} satisfying 
$\varphi_{\ep, \lambda} \in C^1([0, T]; H)$
and 
$\mu_{\ep, \lambda} \in C^1([0, T]; H)$.
\end{lem}
\begin{proof}
We can rewrite \ref{Peplam} as 
\begin{equation}\label{K}\tag{K}
  \begin{cases} 
   \frac{dU}{dt} = L(U)  &\mbox{on}\ [0, T], \\[0.5mm]
   U(0)=U_{0}, 
  \end{cases}
\end{equation}
where 
$U=\left(
    \begin{array}{c}
      \varphi_{\ep, \lambda} \\
      \mu_{\ep, \lambda}
    \end{array}
  \right), \
U_{0}=\left(
    \begin{array}{c}
      \varphi_{0\ep} \\
      \varphi_{0\ep}
    \end{array}
  \right) 
\in H\times H$ and 
\begin{align*}
&L: H\times H \ni 
\left(
    \begin{array}{c}
      \varphi \\
      \mu
    \end{array}
  \right) \\[1mm]
&\hspace{6mm}\mapsto
\left(
    \begin{array}{c}
      -(-\Delta)_{\lambda}\varphi-\varphi-\frac{1}{\ep}a(\cdot)\varphi 
      + \frac{1}{\ep}J\ast\varphi -\frac{1}{\ep}G_{\ep}'(\varphi) 
      + \frac{1}{\ep}\mu
      \\[3mm]
      \frac{1}{\lambda}(-\Delta)_{\lambda}\varphi+\frac{1}{\lambda}\varphi
                  +\frac{1}{\ep\lambda}a(\cdot)\varphi 
                  -\frac{1}{\ep\lambda}J\ast\varphi  
                  +\frac{1}{\ep\lambda}G_{\ep}'(\varphi) 
        -\frac{1}{\ep\lambda}\mu-\frac{1}{\lambda}\mu
                                            -\frac{1}{\lambda}(-\Delta)_{\lambda}\mu
    \end{array}
  \right) 
\\[2mm]
&\hspace{14cm}\in H\times H. 
\end{align*}
Since 
\begin{align*}
&\|(-\Delta)_{\lambda}(\varphi - \psi)\|_{H} 
\leq \frac{1}{\lambda}\|\varphi-\psi\|_{H}, 
\\[1mm]
&\|a(\cdot)(\varphi - \psi)\|_{H} 
\leq \|a\|_{L^{\infty}(\Omega)}\|\varphi-\psi\|_{H}, 
\\[1mm] 
&\|J\ast(\varphi - \psi)\|_{H} 
\leq \|J\|_{L^1(\RN)}\|\varphi-\psi\|_{H}, 
\\[1mm] 
&\|G_{\ep}'(\varphi)-G_{\ep}'(\psi)\|_{H} 
\leq \left(\frac{1}{\ep}+ \|\pi'\|_{L^{\infty}(\mathbb{R})}\right)\|\varphi-\psi\|_{H} 
\end{align*}
for all $\varphi, \psi \in H$, 
the operator $L$ is Lipschitz continuous.     
Thus, by the Cauchy--Lipschitz--Picard theorem,  
there exists a unique classical solution 
$U=
    \left(
    \begin{array}{c}
      \varphi_{\ep, \lambda} \\
      \mu_{\ep, \lambda}
    \end{array}
  \right)
\in C^1([0, T]; H\times H)$ 
of \eqref{K}. 
Therefore we can obtain this lemma. 
\end{proof} 
%
\begin{lem}\label{estimatePeplam}
There exists $\ep_{1} \in (0, 1)$ such that for all $\ep\in(0, \ep_{1})$ 
there exists a constant $C=C(\ep, T)>0$ such that 
\begin{align}
&\|\varphi_{\ep, \lambda}(t)\|^2_{H} 
+ \int_{0}^{t}\|\partial_{t}\varphi_{\ep, \lambda}(s)\|^2_{H}\,ds 
+ \int_{0}^{t}\|\mu_{\ep, \lambda}(s)\|^2_{H}\,ds 
\leq C, \label{eseplam1} \\
&\lambda^2\int_{0}^{t}\|\partial_{t}\mu_{\ep, \lambda}(s)\|^2_{H}\,ds 
\leq C, \label{eseplam2} \\
&\int_{0}^{t}\|(-\Delta)_{\lambda}\mu_{\ep, \lambda}(s)\|^2_{H}\,ds 
\leq C, \label{eseplam3} \\
&\int_{0}^{t}\|(-\Delta)_{\lambda}\varphi_{\ep, \lambda}(s)\|^2_{H}\,ds 
\leq C \label{eseplam4} 
\end{align}
for all $t\in[0, T]$ and all $\lambda>0$.
\end{lem}
\begin{proof}
We see from the first equation in \ref{Peplam} that 
\begin{align}\label{siki1}
\frac{\lambda}{2}\frac{d}{dt}\|\mu_{\ep, \lambda}(t)\|_{H}^2 
+ (\partial_{t}\varphi_{\ep, \lambda}(t), \mu_{\ep, \lambda}(t))_{H} 
+ ((-\Delta)_{\lambda}\mu_{\ep, \lambda}(t), \mu_{\ep, \lambda}(t))_{H} 
+ \|\mu_{\ep, \lambda}(t)\|_{H}^2 =0
\end{align}
and the second equation in \ref{Peplam} yields that 
\begin{align}\label{siki2}
&(\partial_{t}\varphi_{\ep, \lambda}(t), \mu_{\ep, \lambda}(t))_{H}  
\\ \notag 
&= \frac{\ep}{2}\frac{d}{dt}
   (\|(-\Delta)_{\lambda}^{1/2}\varphi_{\ep, \lambda}(t)\|_{H}^2 
                                                          + \|\varphi_{\ep, \lambda}(t)\|_{H}^2) 
   + \frac{1}{2}\frac{d}{dt}\|\sqrt{a(\cdot)}\varphi_{\ep, \lambda}(t)\|_{H}^2 
\\ \notag
&\,\quad -(J\ast\varphi_{\ep, \lambda}(t), \partial_{t}\varphi_{\ep, \lambda}(t))_{H} 
+ \frac{d}{dt}\int_{\Omega} G_{\ep}(\varphi_{\ep, \lambda}(t)) 
+ \ep\|\partial_{t}\varphi_{\ep, \lambda}(t)\|_{H}^2. 
\end{align}
Thus it follows from 
\eqref{tool4}, \eqref{siki1}, \eqref{siki2}, (A8) and the Young inequality that 
\begin{align}\label{siki3}
&\frac{1}{2}\|\sqrt{a(\cdot)}\varphi_{\ep, \lambda}(t)\|_{H}^2 
+ \int_{\Omega}G_{\ep}(\varphi_{\ep, \lambda}(t)) 
+ \ep\int_{0}^{t}\|\partial_{t}\varphi_{\ep, \lambda}(s)\|_{H}^2\,ds 
+ \int_{0}^{t}\|\mu_{\ep, \lambda}(s)\|_{H}^2\,ds 
\\ \notag
&\leq C_{1} 
      + \int_{0}^{t}
        (J\ast\varphi_{\ep, \lambda}(s), \partial_{t}\varphi_{\ep, \lambda}(s))_{H}\,ds 
\\ \notag
&\leq C_{1} 
        + \frac{\|J\|_{L^1(\RN)}^2}{2\ep}
                            \int_{0}^{t}\|\varphi_{\ep, \lambda}(s)\|_{H}^2\,ds 
        + \frac{\ep}{2}\int_{0}^{t}\|\partial_{t}\varphi_{\ep, \lambda}(s)\|_{H}^2\,ds,  
\end{align}
where $C_1>0$.
Here we have from (A4)-(A6), Remark \ref{MYreg}, 
the mean value theorem and the Young inequality that 
\begin{align}\label{siki4}
G_{\ep}(r) 
            &= \frac{1}{2\ep}|r-J_{\ep}^{\beta}(r)|^2 + G(J_{\ep}^{\beta}(r)) 
                + \widehat{\pi}(r) - \widehat{\pi}(J_{\ep}^{\beta}(r)) \\ \notag 
            &\geq \frac{1}{2\ep}|r-J_{\ep}^{\beta}(r)|^2 
                -\frac{\|\pi'\|_{L^{\infty}(\mathbb{R})}}{2}r^2 
                -|\pi(\xi)||r-J_{\ep}^{\beta}(r)| \\ \notag
            &\geq \frac{1}{2\ep}|r-J_{\ep}^{\beta}(r)|^2 
                -\frac{\|\pi'\|_{L^{\infty}(\mathbb{R})}}{2}r^2 
                - \|\pi'\|_{L^{\infty}(\mathbb{R})}|\xi||r-J_{\ep}^{\beta}(r)| \\ \notag 
            &\geq \frac{1}{2\ep}|r-J_{\ep}^{\beta}(r)|^2 
                -\frac{\|\pi'\|_{L^{\infty}(\mathbb{R})}}{2}r^2 
                - 2\|\pi'\|_{L^{\infty}(\mathbb{R})}|r||r-J_{\ep}^{\beta}(r)| \\ \notag
            &\geq -\frac{\|\pi'\|_{L^{\infty}(\mathbb{R})}}{2}r^2 
                - 2\|\pi'\|_{L^{\infty}(\mathbb{R})}\ep r^2 
\end{align}
for all $r\in\mathbb{R}$ and all $\ep>0$, 
where $\xi$ is some constant 
belonging to $[r, J_{\ep}^{\beta}(r)]$ or $[J_{\ep}^{\beta}(r), r]$. 
Therefore, by combining \eqref{siki3}, \eqref{siki4} and (A7), 
there exists $\ep_{1}\in(0, 1)$ such that 
\begin{align*}
&\frac{c_{0}}{4}\|\varphi_{\ep, \lambda}(t)\|_{H}^2
+ \frac{\ep}{2}\int_{0}^{t}\|\partial_{t}\varphi_{\ep, \lambda}(s)\|_{H}^2\,ds 
+ \int_{0}^{t}\|\mu_{\ep, \lambda}(s)\|_{H}^2\,ds 
\\ \notag
&\leq C_{1} 
        + \frac{\|J\|_{L^1(\RN)}^2}{2\ep}
                            \int_{0}^{t}\|\varphi_{\ep, \lambda}(s)\|_{H}^2\,ds  
\end{align*}
for all $t\in[0, T]$, $\ep\in(0, \ep_{1})$ and $\lambda>0$, 
and hence for all $\ep \in (0, \ep_{1})$ 
there exists a constant $C_2=C_2(\ep, T)>0$ such that 
\begin{align}\label{siki5}
\|\varphi_{\ep, \lambda}(t)\|_{H}^2, 
\int_{0}^{t}\|\partial_{t}\varphi_{\ep, \lambda}(s)\|_{H}^2\,ds, 
\int_{0}^{t}\|\mu_{\ep, \lambda}(s)\|_{H}^2\,ds 
\leq C_2
\end{align}
for all $t\in[0, T]$ and all $\lambda>0$.  
The Young inequality and the first equation in \ref{Peplam} yield that 
\begin{align*}
&\lambda^2\|\partial_{t}\mu_{\ep, \lambda}(t)\|_{H}^2 
=\lambda^2(\partial_{t}\mu_{\ep, \lambda}(t), \partial_{t}\mu_{\ep, \lambda}(t))_{H}
\\
&= -\lambda(\partial_{t}\mu_{\ep, \lambda}(t), 
                                                      \partial_{t}\varphi_{\ep, \lambda}(t))_{H} 
   -\frac{\lambda}{2}\frac{d}{dt}
              (\|(-\Delta)_{\lambda}^{1/2}\mu_{\ep, \lambda}(t)\|_{H}^2 
                                                                   + \|\mu_{\ep, \lambda}(t)\|_{H}^2) 
\\ 
&\leq \frac{\lambda^2}{2}\|\partial_{t}\mu_{\ep, \lambda}(t)\|_{H}^2 
        + \frac{1}{2}\|\partial_{t}\varphi_{\ep, \lambda}(t)\|_{H}^2 
        -\frac{\lambda}{2}\frac{d}{dt}
              (\|(-\Delta)_{\lambda}^{1/2}\mu_{\ep, \lambda}(t)\|_{H}^2 
                                                                   + \|\mu_{\ep, \lambda}(t)\|_{H}^2)
\end{align*}
and then we derive from \eqref{tool4}, \eqref{siki5} and (A8) that 
for all $\ep \in (0, \ep_{1})$ there exists a constant $C_{3}=C_{3}(\ep, T)>0$ 
such that 
\begin{align}\label{siki6}
\lambda^2\int_{0}^{t}\|\partial_{t}\mu_{\ep, \lambda}(s)\|_{H}^2\,ds \leq C_{3}
\end{align}
for all $t\in[0, T]$ and all $\lambda>0$. 
We infer from 
\eqref{siki5}, \eqref{siki6}, the first and second equations in \ref{Peplam} 
that   
for all $\ep\in(0, \ep_{1})$
there exists a constant $C_{4}=C_{4}(\ep, T)>0$ such that 
\begin{align}\label{siki7}
\int_{0}^{t}\|(-\Delta)_{\lambda}\mu_{\ep, \lambda}(s)\|^2_{H}\,ds, 
\int_{0}^{t}\|(-\Delta)_{\lambda}\varphi_{\ep, \lambda}(s)\|^2_{H}\,ds 
\leq C_4
\end{align}
for all $t \in [0, T]$ and all $\lambda>0$.

Therefore 
combination of \eqref{siki5}-\eqref{siki7} 
means Lemma \ref{estimatePeplam}.  
\end{proof}
\begin{lem}\label{semisolPep}
Let $\ep_{1}$ be as in Lemma \ref{estimatePeplam}. 
Then for all $\ep \in (0, \ep_{1})$ 
there exist 
$\varphi_{\ep} \in H^1(0, T; H)\cap L^{\infty}(0, T; H)$ 
and $\mu_{\ep} \in L^2(0, T; H)$ 
satisfying \eqref{defsolep1}, \eqref{defsolep2} and \eqref{defsolep3}. 
\end{lem}
\begin{proof}
Let $\ep\in(0, \ep_{1})$.  
The estimates \eqref{eseplam1}-\eqref{eseplam4} yield that 
there exist some functions 
$\varphi_{\ep} \in H^1(0, T; H) \cap L^2(0, T; W)$, 
$\mu_{\ep} \in L^2(0, T; W)$ satisfying  
\begin{align}
&\varphi_{\ep, \lambda} \to \varphi_{\ep} 
\quad \mbox{weakly in}\ L^2(0, T; H), 
\label{Kuri1}\\
&\partial_{t}\varphi_{\ep, \lambda} \to \partial_{t}\varphi_{\ep}  
\quad \mbox{weakly in}\ L^2(0, T; H), 
\label{Kuri2}\\
&(-\Delta)_{\lambda}\varphi_{\ep, \lambda} \to 
-\Delta \varphi_{\ep} \quad \mbox{weakly in}\ L^2(0, T; H), 
\label{Kuri3}\\
&\mu_{\ep, \lambda} \to \mu_{\ep} \quad \mbox{weakly in}\ L^2(0, T; H), 
\label{Kuri4}\\
&\lambda\partial_{t}\mu_{\ep, \lambda} \to 0 \quad \mbox{weakly in}\ L^2(0, T; H),  \label{Kuri5} \\
&(-\Delta)_{\lambda}\mu_{\ep, \lambda} \to 
-\Delta \mu_{\ep} \quad \mbox{weakly in}\ L^2(0, T; H) 
\label{Kuri6}
\end{align}
as $\lambda = \lambda_j \searrow0$. 
We can obtain \eqref{defsolep1} 
by \eqref{Kuri2}, \eqref{Kuri4}, \eqref{Kuri5} and \eqref{Kuri6}. 
Now we show \eqref{defsolep2}. 
To verify \eqref{defsolep2} it suffices to confirm that 
for all $\psi \in C_{\mathrm c}^{\infty}([0, T] \times \Omega)$,  
\begin{align}\label{henbun}
&\int_0^T \Bigl(\int_{\Omega} 
     \bigl(
     \mu_{\ep}(t)-\ep(-\Delta+1)\varphi_{\ep}(t) - a(\cdot)\varphi_{\ep}(t) 
\\ \notag
     &\hspace{20mm}+ J\ast\varphi_{\ep}(t) - G_{\ep}'(\varphi_{\ep}(t)) 
                                                             - \ep\partial_{t}\varphi_{\ep}(t)
     \bigr)\psi(t) \Bigr)\,dt 
= 0. 
\end{align}
From the second equation in \ref{Peplam} 
we infer    
\begin{align}\label{cute1}
&0=\int_{0}^{T} 
(\mu_{\ep, \lambda}(t)-\ep((-\Delta)_{\lambda}+1)\varphi_{\ep, \lambda}(t) 
 - a(\cdot)\varphi_{\ep, \lambda}(t) 
 \\ \notag
     &\hspace{20mm}+ J\ast\varphi_{\ep, \lambda}(t) 
                            - G_{\ep}'(\varphi_{\ep, \lambda}(t)) 
                                 - \ep\partial_{t}\varphi_{\ep, \lambda}(t), \psi(t))_{H}\,dt 
\\[2mm] \notag
&\hspace{0.4em}=\int_{0}^{T} 
(\mu_{\ep, \lambda}(t)-\ep((-\Delta)_{\lambda}+1)\varphi_{\ep, \lambda}(t) 
 - a(\cdot)\varphi_{\ep, \lambda}(t) 
 - \ep\partial_{t}\varphi_{\ep, \lambda}(t), \psi(t))_{H}\,dt 
\\ \notag 
&\,\quad+\int_{0}^{T} (\varphi_{\ep, \lambda}(t), J\ast\psi(t))_{H}\,dt 
- \int_{0}^{T} (G_{\ep}'(\varphi_{\ep, \lambda}(t)), \psi(t))_{H}\,dt
\end{align}
hold. 
Here there exists a bounded domain $B\subset\Omega$ with smooth boundary 
such that 
\begin{align*}
\mbox{supp}\,\psi \subset B\times(0, T).   
\end{align*}
It follows from \eqref{eseplam1}, \eqref{eseplam4} and Lemma \ref{keylemma2} 
that 
\begin{align}\label{tuyoi}
\varphi_{\ep, \lambda} \to \varphi_{\ep} \quad \mbox{in}\ 
L^2(0, T; L^2(B))
\end{align}
as $\lambda = \lambda_j \searrow0$. 
Since $G_{\ep}'=\beta_{\ep}+\pi$ is Lipschitz continuous, 
we see from \eqref{tuyoi} that 
\begin{align}\label{cute2}
\int_{0}^{T} (G_{\ep}'(\varphi_{\ep, \lambda}(t)), \psi(t))_{H}\,dt 
&= \int_0^T \Bigl(\int_{B} G_{\ep}'(\varphi_{\ep, \lambda}(t))\psi(t) \Bigr)\,dt 
\\ \notag 
&\to \int_0^T \Bigl(\int_{B} G_{\ep}'(\varphi_{\ep}(t))\psi(t) \Bigr)\,dt  
\\ \notag
&= \int_{0}^{T} (G_{\ep}'(\varphi_{\ep}(t)), \psi(t))_{H}\,dt
\end{align}
as $\lambda = \lambda_j \searrow0$. 
Thus \eqref{Kuri1}-\eqref{Kuri4}, 
\eqref{cute1} and \eqref{cute2} lead to \eqref{henbun}.  

Next we prove \eqref{defsolep3}. 
Let $\widetilde{\Omega} \subset \Omega$ be 
an arbitrary bounded domain with smooth boundary.  
From \eqref{eseplam1} and Lemma \ref{keylemma} we have 
\begin{equation}\label{itiyou}
J^{1/2}_1\varphi_{\ep, \lambda} \to 
J^{1/2}_1\varphi_{\ep} \quad \mbox{in}\ C([0, T]; L^2(\widetilde{\Omega}))
\end{equation}
as $\lambda = \lambda_j \searrow0$. 
Therefore, 
since $\varphi_{\ep, \lambda}(0)=\varphi_{0\ep}$, 
\eqref{itiyou} yields that 
$$
J^{1/2}_1\varphi_{\ep}(0)=J^{1/2}_1\varphi_{0\ep} 
\quad \mbox{a.e.\ in}\ \widetilde{\Omega}.
$$
Because $\widetilde{\Omega} \subset \Omega$ is arbitrary, we conclude that 
$$
J^{1/2}_1\varphi_{\ep}(0)=J^{1/2}_1\varphi_{0\ep} 
\quad \mbox{a.e.\ in}\ \Omega.
$$ 
Thus, since $J^{1/2}_1\varphi_{0\ep} \in H$, we see that 
$$
J^{1/2}_1\varphi_{\ep}(0)=J^{1/2}_1\varphi_{0\ep} 
\quad \mbox{in}\ H, 
$$
that is, \eqref{defsolep3} holds. 
\end{proof}
\begin{lem}\label{estimatesemisolPep}
Let $\ep_{1}$ be as in Lemma \ref{estimatePeplam} 
and let $\varphi_{\ep}$ and $\mu_{\ep}$ be as in Lemma \ref{semisolPep}. 
Then 
there exists $\ep_{2}\in(0, \ep_{1})$ such that 
for all $\ep\in(0, \ep_{2})$,   
     \begin{equation*}
        \varphi_{\ep} \in H^1(0, T; H)\cap L^{\infty}(0, T; V)\cap L^2(0, T; W),
         \quad \mu_{\ep} \in L^2(0, T; V),  
     \end{equation*}
 and there exists a constant $C=C(T)>0$ 
 such that 
     \begin{align}
     &\|\varphi_{\ep}(t)\|_{H}^2 + \ep\|\varphi_{\ep}(t)\|^2_{V} 
       + \int_{0}^{t}\|\mu_{\ep}(s)\|^2_{V}\,ds  \label{esep1} \\[-0.5mm] 
      &+ \ep\int_{0}^{t}\|(-\Delta+1)\varphi_{\ep}(s)\|^2_{H}\,ds 
       + \ep\int_{0}^{t}\|\partial_{t}\varphi_{\ep}(s)\|^2_{H}\,ds 
     \leq C,  \notag
     \\[1mm]
     &\int_{0}^{t}\|\varphi_{\ep}(s)\|^2_{V}\,ds 
     \leq C, \label{esep2} 
     \\
     &\int_{0}^{t}\|\partial_{t}\varphi_{\ep}(s)\|^2_{V^{*}}\,ds \leq C, \label{esep3} 
     \\
     &\int_{0}^{t}\|\beta_{\ep}(\varphi_{\ep}(s))\|_{H}^2\,ds 
     \leq C \label{esep4}
     \end{align}
 for all $t\in[0, T]$ and all $\ep\in(0, \ep_{2})$. 
\end{lem}
\begin{proof}
Let $\ep\in(0, \ep_{1})$.  
It follows from \eqref{defsolep1} that 
\begin{align}\label{shun1}
(\partial_{t}\varphi_{\ep}(t), \mu_{\ep}(t))_{H} + \|\mu_{\ep}(t)\|_{V}^2 = 0  
\end{align}
and 
\begin{align}\label{shun2}
\frac{1}{2}\frac{d}{dt}\|\varphi_{\ep}(t)\|_{H}^2 
+ (\mu_{\ep}(t), \varphi_{\ep}(t))_{V} = 0,  
\end{align}
and from \eqref{defsolep2} that 
\begin{align}\label{shun3}
&(\partial_{t}\varphi_{\ep}(t), \mu_{\ep}(t))_{H} 
\\ \notag
&= (\partial_{t}\varphi_{\ep}(t), 
               \ep(-\Delta+1)\varphi_{\ep}(t)+a(\cdot)\varphi_{\ep}(t) 
               -J\ast\varphi_{\ep}(t) + G_{\ep}'(\varphi_{\ep}(t)) 
                                                           +\ep\partial_{t}\varphi_{\ep}(t))_{H} 
\\ \notag
&= \frac{d}{dt}
        \left(\frac{\ep}{2}\|\varphi_{\ep}(t)\|_{V}^2
               + \frac{1}{2}\|\sqrt{a(\cdot)}\varphi_{\ep}(t)\|_{H}^2 
               - \frac{1}{2}(\varphi_{\ep}(t), J\ast\varphi_{\ep}(t))_{H}
               + \int_{\Omega} G_{\ep}(\varphi_{\ep}(t)) \right)  
    \\ \notag
    &\,\quad+ \ep\|\partial_{t}\varphi_{\ep}(t)\|_{H}^2              
\\ \notag
&=\frac{d}{dt}
        \left(\frac{\ep}{2}\|\varphi_{\ep}(t)\|_{V}^2
               +\frac{1}{4}\int_{\Omega}\int_{\Omega}
                    J(x-y)(\varphi_{\ep}(x)-\varphi_{\ep}(y))^2dxdy 
               + \int_{\Omega} G_{\ep}(\varphi_{\ep}(t)) \right) 
\\ \notag 
      &\,\quad+ \ep\|\partial_{t}\varphi_{\ep}(t)\|_{H}^2.                
\end{align}
Here we derive 
from the Young inequality, the monotonicity of $\beta_{\ep}$, 
\eqref{defsolep2} and (A7) that 
\begin{align*}
&\frac{c_{0}}{4}\|\varphi_{\ep}(t)\|_{V}^2 + \frac{1}{c_{0}}\|\mu_{\ep}(t)\|_{V}^2 
\\ \notag
&\geq 
(\mu_{\ep}(t), \varphi_{\ep}(t))_{V} 
= (\mu_{\ep}(t), (-\Delta+1)\varphi_{\ep}(t))_{H} \\ \notag
&\geq \ep\|(-\Delta+1)\varphi_{\ep}(t)\|_{H}^2 
     + (c_{0}+\|\pi'\|_{L^{\infty}(\mathbb{R})})\|\varphi_{\ep}(t)\|_{V}^2
     + \int_{\Omega} \varphi_{\ep}(t)\nabla\varphi_{\ep}(t)\cdot\nabla a 
\\ \notag 
     &\,\quad -(J\ast\varphi_{\ep}(t), (-\Delta+1)\varphi_{\ep}(t))_{H} 
         -\|\pi'\|_{L^{\infty}(\mathbb{R})}\|\varphi_{\ep}(t)\|_{V}^2 
     + \frac{\ep}{2}\frac{d}{dt}\|\varphi_{\ep}(t)\|_{V}^2   
\\ \notag 
&\geq \ep\|(-\Delta+1)\varphi_{\ep}(t)\|_{H}^2 
         + \frac{c_{0}}{2}\|\varphi_{\ep}(t)\|_{V}^2 
         - \left(\frac{2}{c_{0}}\|\nabla J\|_{L^1(\RN)}^2  
            + \|J\|_{L^1(\RN)}\right)\|\varphi_{\ep}(t)\|_{H}^2 
\\ \notag 
 &\,\quad+\frac{\ep}{2}\frac{d}{dt}\|\varphi_{\ep}(t)\|_{V}^2,   
\end{align*}
and hence 
\begin{align}\label{shun4}
\|\varphi_{\ep}(t)\|_{V}^2 
&\leq \frac{4}{c_{0}^2}\|\mu_{\ep}(t)\|_{V}^2 
       - \frac{4}{c_{0}}\ep\|(-\Delta+1)\varphi_{\ep}(t)\|_{H}^2 
\\ \notag
      &\,\quad+ \left(\frac{8}{c_{0}^2}\|\nabla J\|_{L^1(\RN)}^2  
            + \frac{4}{c_{0}}\|J\|_{L^1(\RN)}\right)\|\varphi_{\ep}(t)\|_{H}^2 
     -\frac{2}{c_{0}}\ep\frac{d}{dt}\|\varphi_{\ep}(t)\|_{V}^2. 
\end{align}
From \eqref{shun1}-\eqref{shun4} 
we can obtain  
\begin{align*}
&\frac{\ep}{2}\frac{d}{dt}\|\varphi_{\ep}(t)\|_{V}^2
  + \|\mu_{\ep}(t)\|_{V}^2 + \ep\|\partial_{t}\varphi_{\ep}(t)\|_{H}^2
\\ \notag
&+ \frac{1}{2}\frac{d}{dt} 
       \left(\frac{1}{2}\int_{\Omega}\int_{\Omega}
                                   J(x-y)(\varphi_{\ep}(x)-\varphi_{\ep}(y))^2\,dxdy 
              + 2\int_{\Omega}G_{\ep}(\varphi_{\ep}(t)) 
                                 + c_{1}\|\varphi_{\ep}(t)\|_{H}^2\right) 
\\ \notag 
&= -c_{1}(\mu_{\ep}(t), \varphi_{\ep}(t))_{V} 
\\ \notag
&\leq \frac{1}{2}\|\mu_{\ep}(t)\|_{V}^2 + \frac{c_{1}^2}{2}\|\varphi_{\ep}(t)\|_{V}^2 
\\ \notag 
&\leq \left(\frac{1}{2}+\frac{2c_{1}^2}{c_{0}^2}\right)\|\mu_{\ep}(t)\|_{V}^2 
         - \frac{2c_{1}^2}{c_{0}}\ep\|(-\Delta+1)\varphi_{\ep}(t)\|_{H}^2 
\\ \notag 
       &\,\quad+ \left(\frac{4c_{1}^2}{c_{0}^2}\|\nabla J\|_{L^1(\RN)}^2  
                              + \frac{2c_{1}^2}{c_{0}}\|J\|_{L^1(\RN)}
                                                                \right)\|\varphi_{\ep}(t)\|_{H}^2 
       -\frac{c_{1}^2}{c_{0}}\ep\frac{d}{dt}\|\varphi_{\ep}(t)\|_{V}^2,  
\end{align*}
that is, 
\begin{align}\label{shun5}
&\ep\left(\frac{1}{2}+\frac{c_{1}^2}{c_{0}}\right)\frac{d}{dt}\|\varphi_{\ep}(t)\|_{V}^2 
\\ \notag
  &+ \left(\frac{1}{2}-\frac{2c_{1}^2}{c_{0}^2}\right)\|\mu_{\ep}(t)\|_{V}^2 
  + \ep\|\partial_{t}\varphi_{\ep}(t)\|_{H}^2 
  + \frac{2c_{1}^2}{c_{0}}\ep\|(-\Delta+1)\varphi_{\ep}(t)\|_{H}^2 
\\ \notag
&+ \frac{1}{2}\frac{d}{dt} 
       \left(\frac{1}{2}\int_{\Omega}\int_{\Omega}
                                   J(x-y)(\varphi_{\ep}(x)-\varphi_{\ep}(y))^2\,dxdy 
              + 2\int_{\Omega}G_{\ep}(\varphi_{\ep}(t)) 
                                 + c_{1}\|\varphi_{\ep}(t)\|_{H}^2\right) 
\\ \notag 
&\leq \left(\frac{4c_{1}^2}{c_{0}^2}\|\nabla J\|_{L^1(\RN)}^2  
                              + \frac{2c_{1}^2}{c_{0}}\|J\|_{L^1(\RN)}
                                                                \right)\|\varphi_{\ep}(t)\|_{H}^2. 
\end{align}
Here it follows from \eqref{siki4} and (A7) that 
there exists $\ep_{2} \in (0, \ep_{1})$ such that 
\begin{align}\label{shun6}
&\frac{1}{2}\int_{\Omega}\int_{\Omega}
                     J(x-y)(\varphi_{\ep}(x)-\varphi_{\ep}(y))^2\,dxdy 
+ 2\int_{\Omega}G_{\ep}(\varphi_{\ep}(t)) + c_{1}\|\varphi_{\ep}(t)\|_{H}^2 
\\ \notag 
&= \int_{\Omega}a(\cdot)|\varphi_{\ep}(t)|^2 
- (\varphi_{\ep}(t), J\ast\varphi_{\ep}(t))_{H} 
+ 2\int_{\Omega}G_{\ep}(\varphi_{\ep}(t)) + c_{1}\|\varphi_{\ep}(t)\|_{H}^2 
\\ \notag
&\geq \int_{\Omega}(a(\cdot) + c_{1} - \|J\|_{L^1(\RN)})|\varphi_{\ep}(t)|^2 
          - \|\pi'\|_{L^{\infty}(\mathbb{R})}\|\varphi_{\ep}(t)\|_{H}^2 
          - 4\|\pi'\|_{L^{\infty}(\mathbb{R})}\ep\|\varphi_{\ep}(t)\|_{H}^2
\\ \notag 
&\geq \int_{\Omega}(c_{0} + c_{1} - \|J\|_{L^1(\RN)})|\varphi_{\ep}(t)|^2 
          - 4\|\pi'\|_{L^{\infty}(\mathbb{R})}\ep\|\varphi_{\ep}(t)\|_{H}^2  
\\ \notag 
& \geq \frac{c_{0} + c_{1} - \|J\|_{L^1(\RN)}}{2}\|\varphi_{\ep}(t)\|_{H}^2
\end{align}
holds for all $t\in[0, T]$ and all $\ep\in(0, \ep_{2})$. 
Thus \eqref{shun5}, \eqref{shun6} and (A8) yield that  
\begin{align*}
&\ep\left(\frac{1}{2}+\frac{c_{1}^2}{c_{0}}\right)\|\varphi_{\ep}(t)\|_{V}^2
  + \left(\frac{1}{2}-\frac{2c_{1}^2}{c_{0}^2}\right)\int_{0}^{t}\|\mu_{\ep}(s)\|_{V}^2\,ds  
  + \ep\int_{0}^{t}\|\partial_{t}\varphi_{\ep}(s)\|_{H}^2\,ds  
\\ \notag 
&+ \frac{2c_{1}^2}{c_{0}}\ep\int_{0}^{t}\|(-\Delta+1)\varphi_{\ep}(s)\|_{H}^2\,ds  
+ \frac{c_{0} + c_{1} - \|J\|_{L^1(\RN)}}{4}\|\varphi_{\ep}(t)\|_{H}^2
\\ \notag 
&\leq \left(\frac{4c_{1}^2}{c_{0}^2}\|\nabla J\|_{L^1(\RN)}^2  
                              + \frac{2c_{1}^2}{c_{0}}\|J\|_{L^1(\RN)}\right)
                                                         \int_{0}^{t}\|\varphi_{\ep}(s)\|_{H}^2\,ds 
         + C_{1} 
\end{align*}
for all $t\in[0, T]$ and all $\ep\in(0, \ep_{2})$ with some constant $C_{1}>0$. 
Therefore, 
since (A7) leads to the inequalities $\frac{1}{2}-\frac{2c_{1}^2}{c_{0}^2} > 0$ and 
$\frac{c_{0} + c_{1} - \|J\|_{L^1(\RN)}}{4} > 0$, 
there exists a constant $C_{2}=C_{2}(T)>0$ such that 
\begin{align}\label{shun7}
&\|\varphi_{\ep}(t)\|_{H}^2, \ep\|\varphi_{\ep}(t)\|_{V}^2, 
\int_{0}^{t}\|\mu_{\ep}(s)\|_{V}^2\,ds, \\ \notag 
&\ep\int_{0}^{t}\|(-\Delta+1)\varphi_{\ep}(s)\|_{H}^2\,ds, 
\ep\int_{0}^{t}\|\partial_{t}\varphi_{\ep}(s)\|_{H}^2\,ds  
\leq C_2
\end{align}
for all $t\in[0, T]$ and all $\ep\in(0, \ep_{2})$. 
We see from \eqref{shun4}, \eqref{shun7} and (A8) that 
\begin{align}\label{shun8}
\int_{0}^{t} \|\varphi_{\ep}(s)\|_{V}^2\,ds \leq C_3
\end{align}
for all $t\in[0, T]$ and all $\ep\in(0, \ep_{2})$ 
with some constant $C_{3}=C_{3}(T)>0$. 
From \eqref{defF}, \eqref{innerVstar}, \eqref{defsolep1} 
and the Young inequality we have 
\begin{align*}
\|\partial_{t}\varphi_{\ep}(t)\|_{V^{*}}^2 
&= (\partial_{t}\varphi_{\ep}(t), \partial_{t}\varphi_{\ep}(t))_{V^{*}} 
= \langle \partial_{t}\varphi_{\ep}(t), 
                                            F^{-1}\partial_{t}\varphi_{\ep}(t) \rangle_{V^{*}, V} 
\\ \notag
&= -(\mu_{\ep}(t), F^{-1}\partial_{t}\varphi_{\ep}(t))_{V} 
= \langle \partial_{t}\varphi_{\ep}(t), \mu_{\ep}(t) \rangle_{V^{*}, V} 
\\ \notag 
&\leq \frac{1}{2}\|\partial_{t}\varphi_{\ep}(t)\|_{V^{*}}^2 
         + \frac{1}{2}\|\mu_{\ep}(t)\|_{V}^2,  
\end{align*}
and hence \eqref{shun7} yields that 
\begin{align}\label{shun9}
\int_{0}^{t}\|\partial_{t}\varphi_{\ep}(s)\|_{V^{*}}^2\,ds \leq C_{4} 
\end{align}
for all $t\in[0, T]$ and all $\ep\in(0, \ep_{2})$ 
with some constant $C_{4}=C_{4}(T)>0$. 
It follows from \eqref{defsolep2} and the Young inequality that 
\begin{align*}
&\|\beta_{\ep}(\varphi_{\ep}(t))\|_{H}^2 
\\
&= (\beta_{\ep}(\varphi_{\ep}(t)), 
                  \mu_{\ep}(t) - a(\cdot)\varphi_{\ep}(t) 
                                   + J\ast\varphi_{\ep}(t) - \pi(\varphi_{\ep}(t)) 
                                                               - \ep\partial_{t}\varphi_{\ep}(t))_{H} 
\\ 
&\leq \frac{1}{2}\|\beta_{\ep}(\varphi_{\ep}(t))\|_{H}^2 
\\
    &\,\quad+ \frac{5}{2}(\|\mu_{\ep}(t)\|_{V}^2 
                        + (\|a(\cdot)\|_{L^{\infty}(\mathbb{R})}^2 
                        + \|J\|_{L^1(\RN)}^2
                        + \|\pi'\|_{L^{\infty}(\mathbb{R})})\|\varphi_{\ep}(t)\|_{H}^2
    + \ep^2\|\partial_{t}\varphi_{\ep}(t)\|_{H}^2). 
\end{align*}
Thus we derive from \eqref{shun7} that
\begin{align}\label{shun10}
\int_{0}^{t} \|\beta_{\ep}(\varphi_{\ep}(s))\|_{H}^2\,ds \leq C_{5} 
\end{align}
for all $t\in[0, T]$ and all $\ep\in(0, \ep_{2})$ 
with some constant $C_{5}=C_{5}(T)>0$.
 
Therefore combination of \eqref{shun7}-\eqref{shun10} 
means Lemma \ref{estimatesemisolPep}. 
\end{proof}

\begin{prth1.1}
Combination of Lemmas \ref{semisolPep} and \ref{estimatesemisolPep} 
leads to existence and estimates of weak solutions for \ref{Pep} 
for all $\ep \in (0, \ep_{0})$ with some constant $\ep_{0} \in (0, 1)$. 
Now, we confirm that 
the solution $(\varphi_{\ep}, \mu_{\ep})$ of the problem \ref{Pep} 
is unique. 
Assume that 
$(\varphi_{1, \ep}, \mu_{1, \ep})$ and $(\varphi_{2, \ep}, \mu_{2, \ep})$ 
are the solutions of \ref{Pep} 
with the same initial data. 
Then we derive from \eqref{defF}, \eqref{innerVstar} and \eqref{defsolep1} that 
\begin{align}\label{hanpen1}
\frac{1}{2}\frac{d}{dt}\|\varphi_{1, \ep}(t)-\varphi_{2, \ep}(t)\|_{V^{*}}^2 
&= (\partial_{t}\varphi_{1, \ep}(t)-\partial_{t}\varphi_{2, \ep}(t), 
                                                \varphi_{1, \ep}(t)-\varphi_{2, \ep}(t))_{V^{*}} 
\\ \notag
&= \langle \partial_{t}\varphi_{1, \ep}(t)-\partial_{t}\varphi_{2, \ep}(t), 
                          F^{-1}(\varphi_{1, \ep}(t)-\varphi_{2, \ep}(t)) \rangle_{V^{*}, V} 
\\ \notag
&= -(\mu_{1, \ep}(t)-\mu_{2, \ep}(t), 
                                             F^{-1}(\varphi_{1, \ep}(t)-\varphi_{2, \ep}(t)))_{V} 
\\ \notag
&= -\langle \varphi_{1, \ep}(t)-\varphi_{2, \ep}(t), 
                                     \mu_{1, \ep}(t)-\mu_{2, \ep}(t) \rangle_{V^{*}, V} 
\\ \notag
&= -(\varphi_{1, \ep}(t)-\varphi_{2, \ep}(t), 
                                     \mu_{1, \ep}(t)-\mu_{2, \ep}(t))_{H}  
\end{align}
for all $t\in[0, T]$ and all $\ep \in (0, \ep_{0})$, 
and from \eqref{defsolep2} that 
\begin{align}\label{hanpen2}
&(\varphi_{1, \ep}(t)-\varphi_{2, \ep}(t), 
                                     \mu_{1, \ep}(t)-\mu_{2, \ep}(t))_{H}  
\\ \notag
&= \ep\|\varphi_{1, \ep}(t)-\varphi_{2, \ep}(t)\|_{V}^2 
    + \|\sqrt{a(\cdot)}(\varphi_{1, \ep}(t)-\varphi_{2, \ep}(t))\|_{H}^2 
\\ \notag
    &\,\quad- \langle \varphi_{1, \ep}(t)-\varphi_{2, \ep}(t), 
                          J\ast(\varphi_{1, \ep}(t)-\varphi_{2, \ep}(t)) \rangle_{V^{*}, V} 
   \\ \notag
    &\,\quad+ (\beta_{\ep}(\varphi_{1, \ep}(t))-\beta_{\ep}(\varphi_{2, \ep}(t)), 
                                                     \varphi_{1, \ep}(t)-\varphi_{2, \ep}(t))_{H} 
   \\ \notag
    &\,\quad+ (\pi(\varphi_{1, \ep}(t))-\pi(\varphi_{2, \ep}(t)), 
                                                     \varphi_{1, \ep}(t)-\varphi_{2, \ep}(t))_{H} 
     + \frac{\ep}{2}\frac{d}{dt}\|\varphi_{1, \ep}-\varphi_{2, \ep}\|_{H}^2 
\end{align}
for all $t\in[0, T]$ and all $\ep \in (0, \ep_{0})$. 
Hence \eqref{hanpen1}, \eqref{hanpen2} and the monotonicity of $\beta_{\ep}$ 
yield that 
\begin{align}\label{hanpen3}
&\frac{1}{2}\frac{d}{dt}\|\varphi_{1, \ep}(t)-\varphi_{2, \ep}(t)\|_{V^{*}}^2 
+ \|\sqrt{a(\cdot)}(\varphi_{1, \ep}(t)-\varphi_{2, \ep}(t))\|_{H}^2 
+ \frac{\ep}{2}\frac{d}{dt}\|\varphi_{1, \ep}-\varphi_{2, \ep}\|_{H}^2
\\ \notag 
&\leq \langle \varphi_{1, \ep}(t)-\varphi_{2, \ep}(t), 
                          J\ast(\varphi_{1, \ep}(t)-\varphi_{2, \ep}(t)) \rangle_{V^{*}, V} 
+\|\pi'\|_{L^{\infty}(\mathbb{R})}
                                            \|\varphi_{1, \ep}(t)-\varphi_{2, \ep}(t)\|_{H}^2
\end{align}
for all $t\in[0, T]$ and all $\ep \in (0, \ep_{0})$. 
Here we see from the Young inequality that 
\begin{align}\label{hanpen4}
&\langle \varphi_{1, \ep}(t)-\varphi_{2, \ep}(t), 
                          J\ast(\varphi_{1, \ep}(t)-\varphi_{2, \ep}(t)) \rangle_{V^{*}, V} 
\\ \notag
&\leq \frac{\|J\|_{W^{1, 1}(\RN)}^2}{c_{0}}
                                           \|\varphi_{1, \ep}(t)-\varphi_{2, \ep}(t)\|_{V^{*}}^2 
         + \frac{c_{0}}{4\|J\|_{W^{1, 1}(\RN)}^2}
                                       \|J\ast(\varphi_{1, \ep}(t)-\varphi_{2, \ep}(t))\|_{V}^2 
\\ \notag
&\leq \frac{\|J\|_{W^{1, 1}(\RN)}^2}{c_{0}}
                                           \|\varphi_{1, \ep}(t)-\varphi_{2, \ep}(t)\|_{V^{*}}^2 
         + \frac{c_{0}}{2}\|\varphi_{1, \ep}(t)-\varphi_{2, \ep}(t)\|_{H}^2                         
\end{align}
for all $t\in[0, T]$ and all $\ep \in (0, \ep_{0})$. 
Therefore it follows from 
\eqref{defsolep3}, \eqref{hanpen3}, \eqref{hanpen4} and (A7) that 
\begin{align*}
&\frac{1}{2}\|\varphi_{1, \ep}(t)-\varphi_{2, \ep}(t)\|_{V^{*}}^2 
+ \frac{c_{0}}{2}\int_{0}^{t} \|\varphi_{1, \ep}(t)-\varphi_{2, \ep}(t)\|_{H}^2\,ds 
\\ \notag
&\leq \frac{\|J\|_{W^{1, 1}(\RN)}^2}{c_{0}}\int_{0}^{t}
                                      \|\varphi_{1, \ep}(s)-\varphi_{2, \ep}(s)\|_{V^{*}}^2\,ds 
\end{align*}
for all $t\in[0, T]$ and all $\ep \in (0, \ep_{0})$.  
Thus we can prove that the solution of the problem \ref{Pep} is unique. \qed
\end{prth1.1}
 \vspace{10pt}

\section{Existence of solutions to \eqref{P}}\label{Sec5}

This section gives the proof of Theorem \ref{maintheorem2} 
by the following lemma which presents  
Cauchy's criterion for solutions of \ref{Pep}. 
\begin{lem}\label{Cauchy}
Let $\ep_{0}$ and $(\varphi_{\ep}, \mu_{\ep})$ 
be as in Theorem \ref{maintheorem1}. 
Then it holds that 
\begin{align}\label{Cauchycriterion}
\|\varphi_{\ep}-\varphi_{\gamma}\|_{C([0, T]; V^{*})}^2 
+ \int_{0}^{T} \|\varphi_{\ep}(t)-\varphi_{\gamma}(t)\|_{H}^2\,dt 
\leq C(\ep^{1/2}+\gamma^{1/2}) + C\|\varphi_{0\ep}-\varphi_{0\gamma}\|_{V^{*}}^2
\end{align}
for all $\ep, \gamma \in (0, \ep_{0})$ with some constant $C=C(T)>0$.
\end{lem}
\begin{proof}
We see from \eqref{defF}, \eqref{innerVstar} and \eqref{defsolep1} that 
\begin{align}\label{pen1}
\frac{1}{2}\frac{d}{dt}\|\varphi_{\ep}(t)-\varphi_{\gamma}(t)\|_{V^{*}}^2 
&= (\partial_{t}\varphi_{\ep}(t)-\partial_{t}\varphi_{\gamma}(t), 
                                                \varphi_{\ep}(t)-\varphi_{\gamma}(t))_{V^{*}} 
\\ \notag
&= \langle \partial_{t}\varphi_{\ep}(t)-\partial_{t}\varphi_{\gamma}(t), 
                          F^{-1}(\varphi_{\ep}(t)-\varphi_{\gamma}(t)) \rangle_{V^{*}, V} 
\\ \notag
&= -(\mu_{\ep}(t)-\mu_{\gamma}(t), 
                                             F^{-1}(\varphi_{\ep}(t)-\varphi_{\gamma}(t)))_{V} 
\\ \notag
&= -\langle \varphi_{\ep}(t)-\varphi_{\gamma}(t), 
                                     \mu_{\ep}(t)-\mu_{\gamma}(t) \rangle_{V^{*}, V} 
\\ \notag
&= -(\varphi_{\ep}(t)-\varphi_{\gamma}(t), 
                                     \mu_{\ep}(t)-\mu_{\gamma}(t))_{H}  
\end{align}
for all $t\in[0, T]$ and all $\ep, \gamma \in (0, \ep_{0})$, 
and \eqref{defsolep2} yields that 
\begin{align}\label{pen2}
&(\varphi_{\ep}(t)-\varphi_{\gamma}(t), 
                                     \mu_{\ep}(t)-\mu_{\gamma}(t))_{H}  
\\ \notag
&= \ep((-\Delta+1)\varphi_{\ep}(t), \varphi_{\ep}(t)-\varphi_{\gamma}(t))_{H} 
     - \gamma((-\Delta+1)\varphi_{\gamma}(t), 
                                            \varphi_{\ep}(t)-\varphi_{\gamma}(t))_{H} 
    \\ \notag
    &\,\quad+ \|\sqrt{a(\cdot)}(\varphi_{\ep}(t)-\varphi_{\gamma}(t))\|_{H}^2 
    - \langle \varphi_{\ep}(t)-\varphi_{\gamma}(t), 
                          J\ast(\varphi_{\ep}(t)-\varphi_{\gamma}(t)) \rangle_{V^{*}, V} 
   \\ \notag
    &\,\quad+ (\beta_{\ep}(\varphi_{\ep}(t))-\beta_{\gamma}(\varphi_{\gamma}(t)), 
                                                     \varphi_{\ep}(t)-\varphi_{\gamma}(t))_{H} 
   \\ \notag
    &\,\quad+ (\pi(\varphi_{\ep}(t))-\pi(\varphi_{\gamma}(t)), 
                                                     \varphi_{\ep}(t)-\varphi_{\gamma}(t))_{H} 
     + (\ep\partial_{t}\varphi_{\ep}(t) - \gamma\partial_{t}\varphi_{\gamma}(t), 
                                                 \varphi_{\ep}(t)-\varphi_{\gamma}(t))_{H} 
\end{align}
for all $t\in[0, T]$ and all $\ep, \gamma \in (0, \ep_{0})$. 
Thus, by \eqref{pen1} and \eqref{pen2} it holds that 
\begin{align}\label{pen3}
&\frac{1}{2}\frac{d}{dt}\|\varphi_{\ep}(t)-\varphi_{\gamma}(t)\|_{V^{*}}^2 
+ \|\sqrt{a(\cdot)}(\varphi_{\ep}(t)-\varphi_{\gamma}(t))\|_{H}^2 
\\ \notag 
&= \gamma((-\Delta+1)\varphi_{\gamma}(t), 
                                            \varphi_{\ep}(t)-\varphi_{\gamma}(t))_{H} 
   - \ep((-\Delta+1)\varphi_{\ep}(t), \varphi_{\ep}(t)-\varphi_{\gamma}(t))_{H} 
\\ \notag
&\,\quad+ \langle \varphi_{\ep}(t)-\varphi_{\gamma}(t), 
                          J\ast(\varphi_{\ep}(t)-\varphi_{\gamma}(t)) \rangle_{V^{*}, V} 
\\ \notag 
&\,\quad- (\beta_{\ep}(\varphi_{\ep}(t))-\beta_{\gamma}(\varphi_{\gamma}(t)), 
                                                     \varphi_{\ep}(t)-\varphi_{\gamma}(t))_{H} 
\\ \notag
&\,\quad- (\pi(\varphi_{\ep}(t))-\pi(\varphi_{\gamma}(t)), 
                                                     \varphi_{\ep}(t)-\varphi_{\gamma}(t))_{H} 
- (\ep\partial_{t}\varphi_{\ep}(t) - \gamma\partial_{t}\varphi_{\gamma}(t), 
                                                 \varphi_{\ep}(t)-\varphi_{\gamma}(t))_{H} 
\end{align}
for all $t\in[0, T]$ and all $\ep, \gamma \in (0, \ep_{0})$. 
Here we have from the Young inequality that 
\begin{align}\label{pen4}
&\langle \varphi_{\ep}(t)-\varphi_{\gamma}(t), 
                          J\ast(\varphi_{\ep}(t)-\varphi_{\gamma}(t)) \rangle_{V^{*}, V} 
\\ \notag
&\leq \frac{\|J\|_{W^{1, 1}(\RN)}^2}{c_{0}}
                                           \|\varphi_{\ep}(t)-\varphi_{\gamma}(t)\|_{V^{*}}^2 
         + \frac{c_{0}}{4\|J\|_{W^{1, 1}(\RN)}^2}
                                       \|J\ast(\varphi_{\ep}(t)-\varphi_{\gamma}(t))\|_{V}^2 
\\ \notag
&\leq \frac{\|J\|_{W^{1, 1}(\RN)}^2}{c_{0}}
                                           \|\varphi_{\ep}(t)-\varphi_{\gamma}(t)\|_{V^{*}}^2 
         + \frac{c_{0}}{2}\|\varphi_{\ep}(t)-\varphi_{\gamma}(t)\|_{H}^2                         
\end{align}
for all $t\in[0, T]$ and all $\ep, \gamma \in (0, \ep_{0})$, 
and we derive from the monotonicity of $\beta$ that 
\begin{align}\label{pen5}
&-(\beta_{\ep}(\varphi_{\ep}(t))-\beta_{\gamma}(\varphi_{\gamma}(t)), 
                                                     \varphi_{\ep}(t)-\varphi_{\gamma}(t))_{H} 
\\ \notag
&= -(\beta(J_{\ep}^{\beta}(\varphi_{\ep}(t)))
                      -\beta(J_{\gamma}^{\beta}(\varphi_{\gamma}(t))), 
       J_{\ep}^{\beta}(\varphi_{\ep}(t))-J_{\gamma}^{\beta}(\varphi_{\gamma}(t)))_{H} 
\\ \notag
&\,\quad-(\beta_{\ep}(\varphi_{\ep}(t))-\beta_{\gamma}(\varphi_{\gamma}(t)), 
  \ep\beta_{\ep}(\varphi_{\ep}(t))-\gamma\beta_{\gamma}(\varphi_{\gamma}(t)))_{H} 
\\ \notag
&\leq -(\beta_{\ep}(\varphi_{\ep}(t))-\beta_{\gamma}(\varphi_{\gamma}(t)), 
  \ep\beta_{\ep}(\varphi_{\ep}(t))-\gamma\beta_{\gamma}(\varphi_{\gamma}(t)))_{H} 
\end{align}
for all $t\in[0, T]$ and all $\ep, \gamma \in (0, \ep_{0})$. 
Therefore it follows from 
(A7), \eqref{essolep1}, \eqref{essolep4}, 
\eqref{pen3}, \eqref{pen4} and \eqref{pen5} that 
there exists a constant $C_{1}>0$ such that  
\begin{align*}
&\frac{1}{2}\|\varphi_{\ep}(t)-\varphi_{\gamma}(t)\|_{V^{*}}^2 
+ \frac{c_{0}}{2}\int_{0}^{t} \|\varphi_{\ep}(t)-\varphi_{\gamma}(t)\|_{H}^2\,ds 
\\ \notag
&\leq \frac{1}{2}\|\varphi_{0\ep}-\varphi_{0\gamma}\|_{V^{*}}^2 
        + \frac{\|J\|_{W^{1, 1}(\RN)}^2}{c_{0}}\int_{0}^{t}
                                      \|\varphi_{\ep}(s)-\varphi_{\gamma}(s)\|_{V^{*}}^2\,ds 
        + C_{1}(\ep^{1/2} + \gamma^{1/2})
\end{align*}
for all $t\in[0, T]$ and all $\ep, \gamma \in (0, \ep_{0})$, 
and hence we can obtain this lemma. 
\end{proof}
\begin{prth1.2}
We see from Lemma \ref{Cauchy} and (A8) that 
$\{\varphi_{\ep}\}_{\ep \in (0, \ep_{0})}$ 
satisfies Cauchy's criterion in $C([0, T]; V^*) \cap L^2(0, T; H)$, 
and hence there exists a function $\varphi \in C([0, T]; V^*)\cap L^2(0, T; H)$ 
such that 
\begin{align}\label{SK}
\varphi_{\ep} \to \varphi \quad \mbox{in}\ C([0, T]; V^*) 
\ \mbox{and}\ \mbox{in}\ L^2(0, T; H)
\end{align}
as $\ep \searrow 0$.
It follows from \eqref{SK} and (A8) that 
$$
\varphi(0) = \varphi_{0} \quad \mbox{in}\ V^*
$$
and, since $\varphi_{0} \in H$, it holds that 
\begin{align}\label{shokiti}
\varphi(0) = \varphi_{0} \quad \mbox{in}\ H.
\end{align}
By \eqref{essolep1}-\eqref{essolep4} it holds that 
there exist some functions 
\begin{align*}
&v \in H^1(0, T; V^{*}) \cap L^{\infty}(0, T; H) \cap L^2(0, T; V),  
\\
&\mu \in L^2(0, T; V), 
\\
&\xi \in L^2(0, T; H)
\end{align*}
such that 
\begin{align}
&\varphi_{\ep} \to v \quad \mbox{weakly$^{*}$ in}\ L^{\infty}(0, T; H), 
\label{conv1} \\ 
&\partial_{t}\varphi_{\ep} \to v_{t}  \quad \mbox{weakly$^{*}$ in}\ 
L^2(0, T; V^{*}), 
\label{conv2} \\ 
&\ep(-\Delta+1)\varphi_{\ep} \to 0 \quad \mbox{in}\ L^2(0, T; H), 
\label{conv3} \\ 
&\ep\partial_{t}\varphi_{\ep} \to 0 \quad \mbox{in}\ L^2(0, T; H), 
\label{conv4} \\ 
&\varphi_{\ep} \to v \quad \mbox{weakly in}\ L^2(0, T; V), 
\label{conv5} \\ 
&\mu_{\ep} \to \mu \quad \mbox{weakly in}\ L^2(0, T; V), 
\label{conv6} \\ 
&\beta_{\ep}(\varphi_{\ep}) \to \xi \quad \mbox{weakly in}\ L^2(0, T; H) 
\label{conv7} 
\end{align}
as $\ep=\ep_{j}\searrow0$. 
Now, we show that 
\begin{align}\label{varphi=v}
\varphi = v \quad \mbox{a.e.\ on}\ \Omega\times (0, T).
\end{align}
Let $\psi \in L^2(0, T; H)$. 
Then it follows from \eqref{SK} that 
\begin{align}\label{pote1}
\int_{0}^{T} (\varphi_{\ep}(t) - v(t), \psi(t))_{H}\,dt 
\to \int_{0}^{T} (\varphi(t) - v(t), \psi(t))_{H}\,dt 
\end{align}
as $\ep=\ep_{j}\searrow0$. 
On the other hand, we infer from \eqref{conv1} that   
\begin{align}\label{pote2}
\int_{0}^{T} (\varphi_{\ep}(t) - v(t), \psi(t))_{H}\,dt 
\to 0
\end{align}
as $\ep=\ep_{j}\searrow0$. 
Hence from \eqref{pote1} and \eqref{pote2} the identity 
$$
\int_{0}^{T} (\varphi(t) - v(t), \psi(t))_{H}\,dt 
= 0
$$
holds for all $\psi \in L^2(0, T; H)$. 
Thus we obtain \eqref{varphi=v}.  
Consequently, 
from \eqref{conv1}, \eqref{conv2} and \eqref{conv5} 
it holds that  
$\varphi \in H^1(0, T; V^*) \cap L^{\infty}(0, T; H) \cap L^2(0, T; V)$ and 
\begin{align}
&\varphi_{\ep} \to \varphi \quad \mbox{weakly$^{*}$ in}\ L^{\infty}(0, T; H), 
\label{oden1} \\ 
&\partial_{t}\varphi_{\ep} \to \varphi_{t}  
\quad \mbox{weakly$^{*}$ in}\ L^2(0, T; V^{*}), 
\label{oden2} \\ 
&\varphi_{\ep} \to \varphi \quad \mbox{weakly in}\ L^2(0, T; V)  
\label{oden3} 
\end{align}
as $\ep=\ep_{j}\searrow0$. 
From \eqref{defsolep2}, \eqref{SK}, \eqref{conv3}, 
\eqref{conv4}, \eqref{conv6} and \eqref{conv7} we have 
\begin{equation}\label{muxif}
\mu = a(\cdot)\varphi - J\ast\varphi + \xi + \pi(\varphi)  
\quad \mbox{a.e.\ on}\ \Omega\times(0, T).  
\end{equation} 

Next we confirm that 
\begin{equation}\label{betaxi}
\xi = \beta(\varphi) \quad \mbox{a.e.\ on}\ \Omega\times(0, T).
\end{equation}
To verify \eqref{betaxi} it suffices to show that 
\begin{equation}\label{limsup}
\limsup_{\ep\to\infty}\int_0^T (\beta_{\ep}(\varphi_{\ep}(t)), \varphi_{\ep}(t))_{H}\,dt 
\leq \int_0^T (\xi(t), \varphi(t))_{H}\,dt
\end{equation}
(see \cite[Lemma 1.3, p.\ 42]{Barbu1}). 
We see from \eqref{SK} and \eqref{conv7} that 
\begin{align*}
\int_0^T (\beta_{\ep}(\varphi_{\ep}(t)), \varphi_{\ep}(t))_{H}\,dt 
\to \int_{0}^{T} (\xi(t), \varphi(t))_{H}\,dt
\end{align*}
as $\ep=\ep_{j}\searrow0$, 
which means \eqref{limsup}, i.e., \eqref{betaxi}. 

Next we show that there exists a constant $C_1=C_{1}(T)>0$ such that 
\begin{equation}\label{estibeta}
\int_0^t \|\beta(u(s))\|^2_{V}\,ds \leq C_1
\end{equation}
for all $t\in[0, T]$. 
It follows from 
\eqref{essolep1}, \eqref{essolep2}, \eqref{conv6} and \eqref{oden3} that 
there exists a constant $C_2=C_{2}(T) > 0$ such that 
\begin{equation}\label{apmu}
\int_0^t \|\varphi(s)\|^2_{V}\,ds, \int_0^t \|\mu(s)\|^2_{V}\,ds \leq C_2
\end{equation}
for all $t \in [0, T]$.
Therefore, noting that $\mu \in L^2(0, T; V)$ and $\varphi \in L^2(0, T; V)$, 
we derive from \eqref{muxif}, \eqref{betaxi} and \eqref{apmu} 
that $\beta(\varphi) \in L^2(0, T; V)$ and 
\begin{align*}
\int_0^t \|\beta(u(s))\|^2_{V}\,ds 
&= \int_0^t 
      \|\mu(s) - a(\cdot)\varphi(s) + J\ast\varphi(s) - \pi(\varphi(s))\|^2_{V}\,ds \\
&\leq C_{3} 
\end{align*}
for all $t\in[0, T]$ with some constant $C_{3}=C_{3}(T)>0$,
which implies \eqref{estibeta}. 
Hence, by \eqref{defsolep1}, \eqref{shokiti}, \eqref{conv6}, \eqref{oden2},  
\eqref{muxif} and \eqref{betaxi} 
we have proved that $(u, \mu)$ is a solution of \eqref{P}. 
Moreover, from \eqref{essolep1}, \eqref{essolep3}, 
\eqref{oden1}, \eqref{oden2}, \eqref{estibeta} and \eqref{apmu} 
we obtain \eqref{es1}--\eqref{es4}. 

Finally, we can verify that the solution $(u, \mu)$ of the problem \eqref{P} is unique 
in a similar way to the proof of Theorem \ref{maintheorem1}.
\qed
\end{prth1.2}


\section{Energy estimate for \eqref{P}}\label{Sec6}

\begin{prth1.3}
We see from \eqref{defsol1} and \eqref{defsol2} that    
\begin{align*}
\frac{d}{dt}E(t) 
&= \frac{d}{dt}\left(\frac{1}{2}\|\sqrt{a(\cdot)}\varphi(t)\|_{H}^2 
                         - \frac{1}{2}(\varphi(t), J\ast\varphi(t))_{H}
                         +\int_{\Omega}G(\varphi(t)) \right)                         
\\ 
&=\langle \varphi_{t}(t), 
                      a(\cdot)\varphi(t)-J\ast\varphi(t) + G'(\varphi(t))  
                                                                                 \rangle_{V^{*}, V} 
\\ 
&=\langle \varphi_{t}(t), \mu(t) \rangle_{V^{*}, V} 
\\ 
&= - \|\mu(t)\|_{V}^2, 
\end{align*}
which implies Theorem \ref{maintheorem3}. \qed
\end{prth1.3}


\section{Error estimate between \eqref{P} and {\rm \ref{Pep}}}\label{Sec7}

\begin{prth1.4}
We can obtain Theorem \ref{maintheorem4} by Lemma \ref{Cauchy}. \qed
\end{prth1.4}

%
 
\end{document}